\documentclass[a4paper,11pt,reqno]{amsart}
 \usepackage{enumerate}
 \usepackage{a4wide}
\usepackage{graphics}
 \usepackage{amssymb, amsmath}
 \usepackage{mathrsfs}
 \usepackage{amscd}
 \usepackage[active]{srcltx}
 \usepackage{verbatim}
 \usepackage[colorlinks,linkcolor={blue},citecolor={red},urlcolor={blue}]{hyperref}
\usepackage{comment}

\theoremstyle{plain}
\newtheorem{theorem}{Theorem}[section]

\newtheorem{lemma}[theorem]{Lemma}
\newtheorem{proposition}[theorem]{Proposition}

\newtheorem{definition}[theorem]{Definition}

\newtheorem*{definition*}{Definition}

\theoremstyle{remark}
\newtheorem{remark}[theorem]{Remark}
\newtheorem{example}[theorem]{Example}

\newtheorem*{claim*}{Claim}
\newtheorem*{remark*}{Remark}
\newtheorem*{example*}{Example}
\newtheorem*{notation*}{Notation}

\numberwithin{equation}{section}

\def\R{{\mathbb R}}

\def\Q{{\mathbb Q}}

\newcommand{\eps}{\varepsilon}
\renewcommand{\phi}{\varphi}

\newcommand{\dd}{\mathrm{d}}
\newcommand{\ddd}{\; \mathrm{d}}

\newcommand{\ip}[1]{\langle {#1}\rangle}

\newcommand{\norm}[1]{\| {#1}\|}
\newcommand{\abs}[1]{\left\vert {#1}\right\vert}

\DeclareMathOperator{\Ric}{Ric}

\DeclareMathOperator{\Hess}{Hess}

\newcommand{\cH}{\mathcal{H}}

\newcommand{\cW}{\mathcal{W}}

\newcommand{\cS}{\mathcal{S}}

\newcommand{\cV}{\mathcal{V}}

\newcommand{\cX}{\mathcal{X}}
\newcommand{\cY}{\mathcal{Y}}
\newcommand{\cZ}{\mathcal{Z}}

\newcommand{\cA}{\mathcal{A}}

\newcommand{\CE}{\mathcal{CE}}

\newcommand{\cP}{\mathscr{P}}
\newcommand{\PX}{\cP(\cX)}
\newcommand{\PXs}{\cP_*(\cX)}

\renewcommand{\tilde}{\widetilde}

\DeclareMathOperator{\ent}{Ent}
\DeclareMathOperator{\HJ}{HJ}

\begin{document}

\title[Super Ricci flows for weighted graphs] {Super Ricci flows for weighted graphs}

\author{Matthias Erbar}
\author{Eva Kopfer}
\address{
University of Bonn\\
Institute for Applied Mathematics\\
Endenicher Allee 60\\
53115 Bonn\\
Germany}
\email{erbar@iam.uni-bonn.de}
\email{eva.kopfer@iam.uni-bonn.de}



 \begin{abstract}
   We present a notion of super Ricci flow for time-dependent finite
   weighted graphs. A challenging feature is that these flows
   typically encounter singularities where the underlying graph
   structure changes. Our notion is robust enough to allow the flow to
   continue past these singularities. As a crucial tool for this purpose we study the heat flow on such singular time-dependent
   weighted graphs with changing graph structure. We then give several
   equivalent characterizations of super Ricci flows in terms of a
   discrete dynamic Bochner inequality, gradient and transport
   estimates for the heat flow, and dynamic convexity of the entropy
   along discrete optimal transport paths. The latter property can be
   used to show that our notion of super Ricci flow is consistent with
   classical super Ricci flows for manifolds (or metric measure
   spaces) in a discrete to continuum limit.
 \end{abstract}

\date\today

\maketitle
 
\tableofcontents
  
\section{Introduction}
\label{sec:intro}

The main purpose of the present paper is to identify a natural time
evolution of weighted graphs that can be considered as a discrete
analogue of (super-)Ricci flow. Its second purpose is a study of the
heat equation on time-dependent weighted graphs in a general
setting. The latter will serve as a tool to give robust
characterizations of discrete super Ricci flows, but might also be of
independent interest. Before we enter the discrete setting, let us
recall the classical notion of (super-)Ricci flow for manifolds and
recent developments that motivate our work.

\smallskip

A smooth manifold $M$ equipped with a one-parameter family
$(g_t)_{t\in I}$ of Riemannian metrics evolves as a \emph{Ricci flow}
if $\Ric_{g_t}=-\frac12\partial_tg_t$ for all $t\in I$. It is called a
\emph{super Ricci flow} if instead only
$\Ric_{g_t}\geq-\frac12\partial_tg_t$ is satisfied as an inequality
between quadratic forms, i.e.~super Ricci flows are `super solutions'
to the Ricci flow equation.

Since the seminal work of Hamilton \cite{H82,H95} and Perelman
\cite{P02,P03,P03b}, see also \cite{CZ06,KL08,MT07}, Ricci flow has
received a lot of attention and has become a powerful tool in many
applications. A challenging feature is that the flow typically
develops singularities in finite time. Currently, a lot of activity is
being devoted to extend the scope of Ricci flows beyond the setting of
smooth manifolds. A major challenge is to define and analyze flows
that pass through singularities where dimension and/or topological
type changes. Among the exciting recent contributions we mention the
work of Bamler, Kleiner and Lott \cite{KL17,BK17} constructing
canonical Ricci flows through singularities in dimension 3 as the
limit of flows with surgery and the work of Haslhofer and Naber
\cite{HN15} characterizing Ricci flows in terms of functional
inequalities on the path space, see also Cheng and Thalmeier
\cite{CT17}. Sturm \cite{sturm2015} introduced a synthetic definition
of super Ricci flow that applies to time-dependent metric measure
spaces using optimal transport. Here, the crucial observation is that
for a smooth family of Riemannian manifolds to be a super Ricci flow
is equivalent to \emph{dynamic convexity} of the Boltzmann entropy
along geodesics in the space of probability measures equipped with the
(time-dependent) $L^2$-Kantorovich distance (see
Sec.~\ref{sec:characterize} for a definition). The latter property is
meaningful when the manifold is replaced with a time-dependent metric
measure space and serves as a synthetic definition of super Ricci
flow.

In the case of a static Riemannian metric, the super Ricci flow
equation becomes $\Ric_g\geq 0$ and the notion of dynamic convexity
reduces to convexity of the entropy along geodesics in the
Kantorovich distance, the property used as a synthetic
definition of lower Ricci curvature bounds in the celebrated works of
Lott, Sturm and Villani \cite{S06,LV09}.

\smallskip

In view of the powerful applications of Ricci flow, it seems desirable
to develop a similar concept for discrete spaces, for instance as a
natural way of deforming a given space to a simpler
object. Unfortunately, the approach of Sturm \cite{sturm2015} does not
apply in this situation since the $L^2$ Kantorovich distance is
degenerate if the underlying space is discrete in the sense that it
does not admit geodesics. The main objective of the present article is
to develop a notion of super Ricci flow that applies to discrete
spaces, namely to time-dependent weighted graphs. In order to
circumvent the non-existence of geodesics, we will replace the
Kantorovich distance by a different distance $\cW$ on the space of
probability measures, constructed in \cite{Ma11}, that is well-adapted
to the discrete setting. In the case of a static weighted graph (or
Markov chain) this distance has been used successfully in \cite{EM12}
to define a notion of lower Ricci curvature bounds in the spirit of
the theory of Lott, Sturm and Villani via convexity of the entropy
along $\cW$-geodesics. Here, in the time-dependent case, super Ricci
flow will be defined via dynamic convexity of the entropy.

\smallskip

\begin{figure}[h]
  \centering
    \resizebox{4in}{!}{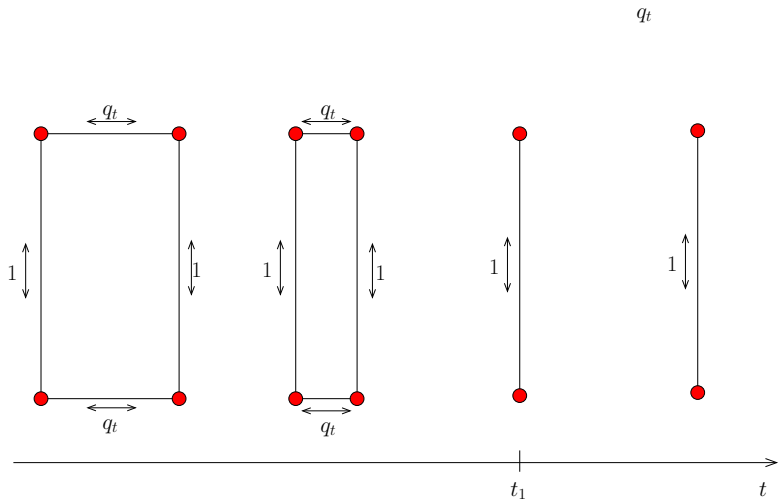}
  \caption{Example of a discrete super Ricci flow}
  \label{fig:ex-intro}
\end{figure}

As in the continuous case, our discrete super Ricci flows will
typically produce singularities in finite time. A simple example is
depicted in Figure \ref{fig:ex-intro}. Here, several vertices collapse
and the weights of the connecting edges explode at the singularity
$t_1$ like $q_t=1/(t_1-t)$ for instance. This can be seen in analogy
to the continuous example of $S^2\times \mathbb{T}^2$ equipped with
the product $\big((1-2t)g_{S^2}\big)\otimes g_{\mathbb{T}^2}$ of the
scaled round and the flat metric and collapsing to $\mathbb{T}^2$ at
$t=\frac12$. An important feature of our approach is that it allows to
define discrete super Ricci flows \emph{through} such
singularities. In fact, we will show that discrete super Ricci flows
can be characterized equivalently via a discrete dynamic Bochner
inequality and via gradient and transport estimates for the heat
flow. These latter characterizations hold consistently across singular
times where the graph structure changes. To this end, we perform a
detailed analysis of the heat flow on general time-dependent weighted
graphs allowing for a variety of singular phenomena such as collapse
and spawning of vertices, or deletion and creation of edges. In
particular, we establish existence and uniqueness of the heat flow.

Finally, this provides a second motivation for our investigation of
the discrete setting as a 'sandbox' to develop methods to be used
eventually also in the technically more challenging setting of
continuous singular time-dependent spaces and (super) Ricci flows.
For instance, the analysis of the heat flow on time-dependent metric
measure space, initiated in \cite{sturm2016}, currently cannot deal
with the singularities at which the base space changes.

\subsection{Robust characterizations of super Ricci flows}
\label{sec:characterize}
Before we describe our main results in more detail, let us briefly
recall several robust characterizations of classical super Ricci
flows in terms of the heat flow and optimal transport, as they will serve as a guideline for the discrete setting.

Let $(g_t)_{t\in I}$ be a smooth family of Riemannian metrics on a (compact) manifold $M$. We denote by $\Delta_t$ the Laplace--Beltrami operator associated
with $g_t$. The heat flow is given by the propagator
$P_{t,s}\bar\psi$, defined for $s\leq t$ as the solution to
the heat equation $\partial_t\psi=\Delta_t\psi$ with initial condition
$\psi(\cdot,s)=\bar\psi$. By duality, we define the heat flow on
probability measures given by the propagator $\hat P_{t,s}\mu$
characterized via
$\int \psi \ \dd (\hat P_{t,s}\mu)=\int P_{t,s}\psi\ \dd\mu$.

The $L^2$-Kantorovich distance on the space probability measures $\cP(M)$ is given by
\begin{align*}
  W_{2,t}(\mu,\nu)^2 = \inf_\pi \int d_t(x,y)^2\ \dd\pi(x,y)\;,
\end{align*}
where the infimum is taken over all couplings of $\mu$ and $\nu$ and $d_t$
is the Riemannian distance. Finally, denote by
$\cH_t(\mu)=\int\rho\log\rho\dd\text{vol}_{g_t}$ for
$\mu=\rho\text{vol}_{g_t}$ the Boltzmann entropy. The connection between
these objects is that the (dual) heat flow evolves as the gradient
flow of the entropy w.r.t.~the Kantorovich distance.

Now, the super Ricci flow equation
\begin{align}\label{eq:srf-intro}
  \Ric_{g_t} \geq -\frac12 \partial_t g_t
\end{align}
is equivalent to any of the following properties:

\begin{itemize}
\item[(I)] \emph{dynamic Bochner inequality}: for all smooth functions $\psi$ on $M$ and $t\in I$:
  \begin{align*}
    \Gamma_{2,t}(\psi)\geq \frac12\partial_t\Gamma_t(\psi)\;,
  \end{align*}
  where $\Gamma_t(\psi):=|\nabla\psi|^2_{g_t}$ and
  $\Gamma_{2,t}(\psi):=\frac12
  \Delta_t\big(|\nabla\psi|^2_{g_t}\big)-\ip{\nabla\psi,\nabla\Delta_t\psi}_{g_t}$
  are the \emph{carr\'e du champs} operators associated to the
  Laplace--Beltrami operator $\Delta_t$,
\item[(II)] \emph{gradient estimate}: for all smooth functions $\psi$ and $s\leq t$:
  \begin{align*}
    \Gamma_t(P_{t,s}\psi) \leq P_{t,s}\Gamma_s(\psi)\;,
  \end{align*}
\item[(III)] \emph{transport estimates}: for all probability measures $\mu,\nu$ and $s\leq t$:
  \begin{align*}
    W_{2,s}(\hat P_{t,s}\mu,\hat P_{t,s}\nu) \leq W_{2,t}(\mu,\nu)\;,
  \end{align*}
\item[(IV)] \emph{dynamic convexity of entropy}: for all $t$ and all geodesics $(\mu^a)_{a\in[0,1]}$ in $(\cP(M),W_{2,t})$:
  \begin{align*}
    \partial^+_a\cH_t(\mu^{1-})-\partial_a^-\cH_t(\mu^{0+}) \geq -\frac12\partial_t^-W_{t-}(\mu^0,\mu^1)^2\;.
  \end{align*}
\end{itemize}

Here and in the sequel we denote by $\partial_a^\pm f(a\pm)$ the
upper/lower right/left derivative of $f$ at $a$, i.e.~for instance
\begin{align*}
  \partial_a^+ f(a+):=\limsup_{b\searrow a} \frac{f(b)-f(a)}{b-a}\;.
\end{align*}
The connection between (I) and \eqref{eq:srf-intro} stems immediately
from the Bochner identity
$\Gamma_{2,t}(\psi)=\Ric_{g_t}[\nabla\psi]+||\text{Hess}_{g_t}\psi||_{\text{HS}}^2$. (I)
and (II) are connected via a classical interpolation argument. The
characterization (III) in terms of non-expansion of the transport
distance under the heat flow was observed by McCann and Topping
\cite{MT10}. Characterization (IV) was established in \cite{sturm2015}
and should be thought of as a quantified formulation of convexity in
terms of the increase of the first derivative. In the static case it
reduces to plain convexity of the entropy along geodesics
characterizing non-negative Ricci curvature, see \cite{vRS05,CEMS01}.

As already mentioned, the advantage of the these characterizations is
their robustness, i.e.~ that they remain meaningful in a non-smooth
setting. For instance (I), (II) can be formulated for a family of
time-dependent Dirichlet forms. (IV) requires only the structure of a
time-dependent metric measure space $(X_t,d_t,m_t)_{t\in I}$. Sturm
and the second author \cite{sturm2016} proved that the equivalence of
(I) - (IV) holds in the setting of metric measure spaces, at least
under stringent regularity conditions (namely, $X_t\equiv X$ is
independent of $t$, a curvature-dimension bound RCD$(K,\infty)$ holds
uniformly in time, and Lipschitz controls on $d_t$ and $m_t$).

\smallskip

In this article, in the setting of time-dependent weighted graphs, we
will obtain similar equivalent characterization (I)-(IV), where the
carr\'e du champs operators and the transport distance are replaced
with suitable discrete counterparts, and where we allow for changing
graph structure.

\subsection{Main results}
\label{sec:main}
Let us now discuss the content of this article in more detail. 

We will consider a time-dependent family of \emph{Markov triples}
$(\cX_t,Q_t,\pi_t)_{t\in [0,T]}$.  Here for each $t$, $\cX_t$ is a
finite set, $\pi_t$ is a strictly positive probability measure on
$\cX_t$, and $Q_t:\cX_t\times\cX_t\to\R$ is a kernel giving the
transition rates of a continuous time Markov chain with the convention
that $Q_t(x,y)\geq0$ for $x\neq y$ and
$Q_t(x,x)=-\sum_{y\neq x}Q_t(x,y)$.  We will assume that $Q_t$ is
reversible, i.e.~the detailed balance condition holds:
\begin{align*}
  Q_t(x,y)\pi_t(x)=Q_t(y,x)\pi_t(y)\quad \forall x,y\in \cX_t\;.
\end{align*}
Equivalently, we can consider the family of weighted graphs
$(\cX_t,w_t,\pi_t)$, where $\cX_t$ is the set of vertices, $\pi_t$ is
the vertex weight, and the symmetric function
$w_t(x,y):=Q_t(x,y)\pi_t(x)$ is the edge-weight and the set of edges
is given by $\cV_t=\big\{\{x,y\}:w_t(x,y)>0\big\}$.  We will also
assume that $Q_t$ is irreducible, i.e.~the associated graph is
connected.

We allow the graph structure to change at a finite number of
times. More precisely, we will assume that there exists a partition
$0=t_0<t_1<\dots<t_n=T$ such that $(\cX_t,\cV_t)\equiv (\cX_i,\cV_i)$
for all $t\in I_i=(t_i,t_{i+1})$ and all $i=0,\dots,n-1$. During the
intervals $I_i$ we assume that $t\mapsto \pi_t$ is Lipschitz and that
$t\mapsto Q_t$ is locally Lipschitz with limits existing in
$[0,+\infty]$ as we approach the singular times, i.e.~$t\searrow t_i$
and $t\nearrow t_{i+1}$. If the limit of $Q_t(x,y)$ is $+\infty$, we
assume moreover, that the accumulated transition rate explodes, i.e.~
\begin{align}\label{eq:accumulated}
\int^{t_{i+1}}Q_t(x,y)\dd t=+\infty\quad \text{resp.~} \quad\int_{t_{i}}Q_t(x,y)\dd t=+\infty\;.
\end{align}
Moreover, the limiting weights are assumed to be compatible with the
weighted graph structure at singular times $t_i$. For a precise
statement of our assumptions see Section \ref{sec:topo}.

The interpretation is that the graph structure can change at singular
times $t_i$ due to different phenomena:
\begin{itemize}
\item edges can disappear (resp.~appear), corresponding to
  $w_t(x,y)\to0$ as $t\nearrow t_{i}$ (resp.~$t\searrow t_i$),
\item two vertices $x,y$ can \emph{collapse}, this happens if
  $w_t(x,y)\to\infty$ as $t\nearrow t_{i}$,
\item a vertex can \emph{spawn} new vertices (same as collapse but
  backwards in time).
\end{itemize}
Let us denote for $z\in\cX_{t_{i+1}}$ by $C_z\subset \cX_i$ the set of
vertices that collapse onto $z$ at $t_{i+1}$. Similarly, let
$S_z\subset \cX_{i+1}$ denote the set of vertices spawned by $z$.

\smallskip

Our first main result (see Thm.~\ref{thm:heat-eq} below) concerns the
existence and uniqueness of solutions to the (dual) heat equation in
this general setting. To this end, we introduce the discrete Laplacian
$\Delta_t$ and dual Laplacian $\hat\Delta_t$ associated with
$(\cX_t,Q_t,\pi_t)$ acting on functions $\psi,\sigma\in \R^{\cX_t}$
via
\begin{align*}
  \Delta_t\psi(x)&:=\sum_{y\in\cX_t}\big[\psi(y)-\psi(x)\big]Q_t(x,y)\;,\\
\hat\Delta_t\sigma(x)&:=\sum_{y\in\cX_t}\big[Q_t(y,x)\sigma(y)-Q_t(x,y)\sigma(x)\big]\;.
\end{align*}
For $0\leq s<t\leq T$, let us define space-time during the
interval $[s,t]$ by setting
\begin{align*}
  \cS_{s,t}:=\big\{(r,x) : r\in[s,t],~x\in\cX_r\big\}\;.
\end{align*}

\begin{theorem}\label{thm:heat-eq-intro}
  Given $s\in[0,T]$ and $\bar\psi\in\R^{\cX_s}$ there exist a unique
  $\psi:\cS_{s,T}\to\R$ such that:
  \begin{itemize}
  \item $\psi(s,\cdot)=\bar\psi$, the map $t\mapsto \psi(t,\cdot)$ is
    differentiable on each $I_i=(t_i,t_{i+1})$ and satisfies
   \begin{align}\label{eq:heat-intro}
\partial_t\psi=\Delta_t\psi \quad\text{on } I_i\times\cX_i\;,
    \end{align}

  \item for all $z\in \cX_{t_{i}}$, $x\in S_z$ and $y\in C_z$ we have
    \begin{align}\label{eq:heat-eq-cont-intro}
      \psi(t_i,z)=\lim_{t\downarrow t_i}\psi(t,x)
=\lim_{t\uparrow t_i}\psi(t,y)\;.
 \end{align}
\end{itemize}
Given $t\in[0,T]$ and $\bar\sigma\in\R^{\cX_t}$ there exist a unique
$\sigma:\cS_{0,t}\to[0,\infty)$ such that:
  \begin{itemize}
  \item $\sigma(t,\cdot)=\bar\sigma$, the map
$s\mapsto \sigma(s,\cdot)$ is differentiable on each $I_i=(t_i,t_{i+1})$ and satisfies 
\begin{align}\label{eq:dual-heat-intro}
\partial_s\sigma=-\hat\Delta_s\sigma\quad\text{on } I_i\times\cX_i\;,
\end{align}
  \item for all $z\in \cX_{t_i}$ we have 
    \begin{align}\label{eq:heat-eq-adj-cont-intro}
     \sigma(t_i,z)= \sum_{x\in S_z}\lim_{s\downarrow t_i}\sigma(s,x)
 =\sum_{y\in C_z}\lim_{s\uparrow t_i}\sigma(s,y) \;.
    \end{align}
  \end{itemize}  
\end{theorem}
We define the heat propagator $P_{t,s}:\R^{\cX_s}\to\R^{\cX_t}$ and
dual heat propagator $\hat P_{t,s}:\R^{\cX_t}\to\R^{\cX_s}$ by setting
$ P_{t,s}\bar\psi=\psi(t,\cdot)$,
$\hat P_{t,s}\bar\sigma = \sigma(s,\cdot)$. Note that the dual heat
equation is interpreted as running backwards in time. This is natural
in view of the following duality relation. Interpreting the Euclidean
scalar product $\ip{\psi,\sigma}$ as the integral of $\psi$ against a
(signed) measure $\sigma$, we have that
\begin{align}\label{eq:dual-intro}
  \ip{P_{t,s}\psi,\sigma} = \ip{\psi,\hat P_{t,s}\sigma}\;.
\end{align}

Existence and uniqueness of solutions to \eqref{eq:heat-intro} and
\eqref{eq:dual-heat-intro} on the intervals $I_i$ is of course
guaranteed by standard theory of ODEs. The first non-trivial aspect of
the previous theorem is that the solution has a well-defined limit as
we approach singular times. Here the assumption \eqref{eq:accumulated}
will be crucial, which ensures that during a collapse the solution to
the heat equation already equilibrates before the singular time on the
group of collapsing vertices and thus leads to
\eqref{eq:heat-eq-cont-intro}, similarly for the dual heat equation
and spawning events. The second non-trivial aspect is that the
solution can be continued from singular times. Here, the dual equation
\eqref{eq:dual-heat-intro} starting from non-singular times will be
used to construct the solution to \eqref{eq:heat-intro} and vice-versa
exploiting the duality \eqref{eq:dual-intro}.

\smallskip

In order to state our second main result on the characterization of
discrete super Ricci flows, we need to introduce discrete analogues of
the optimal transport distance and the carr\'e du champs operators.

For each $t$ we consider the \emph{discrete transport distance}
$\cW_t$ between probability measures $\mu^0,\mu^1\in\cP(\cX_t)$ given
by
\begin{align*}
  \cW_t(\mu^0,\mu^1)^2 = \inf_{\mu,V}\left\{\int_0^1\frac12\sum_{x,y\in\cX_t}\frac{|V_{a}(x,y)|^2}{\Lambda(\mu^a)_t(x,y)} \dd a\right\}\;,
\end{align*}
where the infimum runs over all sufficiently regular curves
$\mu:[0,1]\to\cP(\cX_t)$ connecting $\mu^0$ and $\mu^1$, and
$V:[0,1]\to\R^{\cX_t\times\cX_t}$ satisfying the \emph{discrete
  continuity equation}
\begin{align*}
  \frac{\dd}{\dd a}\mu^a(x) +\frac12\sum_{y\in\cX_t}\big[V^a(x,y)-V^a(y,x)\big]=0\;, 
\end{align*}
and we write
$\Lambda(\mu)_t:=\Lambda\big(\mu(x)Q_t(x,y),\mu(y)Q_t(y,x)\big)$,
where $\Lambda(s,t):=\int_0^1s^\alpha t^{1-\alpha}\dd \alpha$ denotes
the \emph{logarithmic mean} of $s,t\geq0$. This distance associated to
a Markov triple has been introduced in \cite{Ma11} and can be thought
of as a discrete analogue of the Benamou--Brenier formula for the
$L^2$-Kantorovich distance.

Moreover, we introduce for $\psi\in\R^{\cX_t}$ and $\mu\in\cP(\cX_t)$
the \emph{integrated carr\'e du champs operator}
\begin{align*}
  \Gamma_t(\mu,\psi) = \ip{\nabla\psi,\nabla\psi\cdot\Lambda(\mu)_t}\;,
\end{align*}
where $\nabla\psi(x,y)=\psi(y)-\psi(x)$ denotes the discrete gradient,
and the multiplication with $\Lambda(\mu)_t$ is understood
componentwise in $\R^{\cX_t\times\cX_t}$. We also introduce an
integrated iterated carr\'e du champs operator
$\Gamma_{2,t}(\mu,\psi)$, see Section \ref{sec:BA}. These quantities
should be thought of as discrete analogues of
\begin{align*}
  \int\Gamma_t(\psi)\dd\mu\;, \int\Gamma_{2,t}(\psi)\dd\mu\;,
\end{align*}
where $\Gamma_t, \Gamma_{2,t}$ are the carr\'e du champs operators
associated to the Laplacian $\Delta_t$ in the continuous setting,
c.f.~Section \ref{sec:characterize}. Finally, let us denote by
\begin{align*}
  \cH_t(\mu)=\sum_{x\in\cX_t}\log\frac{\mu(x)}{\pi_t(x)}\mu(x)
\end{align*}
the relative entropy of $\mu\in\cP(\cX_t)$ w.r.t.~the reference
measure $\pi_t$.
Our second main result (see Theorem \ref{thm:srf-equiv}) is the following:

\begin{theorem}\label{thm:srf-equiv-intro}
  Let $(\cX_t,Q_t,\pi_t)_{t\in[0,T]}$ be a time-dependent Markov
  triple satisfying \eqref{eq:rate-int-infty2} and
  \eqref{eq:rate-int-infty3}. Then the following are equivalent:
\begin{enumerate}
\item [(I)] The \emph{dynamic Bochner inequality}
\begin{align}\label{eq:Bochner-intro}
\Gamma_{2,t}(\mu,\psi)\geq\frac12\partial_t\Gamma_t(\mu,\psi)
\end{align}
holds for a.e.~$t\in[0,T]$ and all $\mu\in\cP(\cX_t)$, $\psi\in\R^{\cX_t}$.
\item[(II)] The \emph{gradient estimate}
\begin{align}\label{eq:grad-est-intro}
\Gamma_t(\mu,P_{t,s}\psi)\leq\Gamma_s(\hat P_{t,s}\mu,\psi)
 \end{align}
  holds for all $0\leq s\leq t\leq T$ and all $\mu\in\cP(\cX_t)$, $\psi\in\R^{\cX_s}$.
\item[(III)] The \emph{transport estimate}
\begin{align}\label{eq:transport-intro}
\cW_s(\hat P_{t,s}\mu,\hat P_{t,s}\nu)\leq \cW_t(\mu,\nu)
 \end{align}
  holds for all $0\leq s \leq t\leq T$ and all $\mu,\nu\in\cP(\cX_t)$.
\item[(IV)] The \emph{entropy is dynamically convex}, i.e.~for
  a.e.~$t\in[0,T]$ and all $\cW_t$-geodesics $(\mu^a)_{a\in[0,1]}$
\begin{align}\label{eq:dyn-convex-intro}
\partial_a^+\cH_t(\mu^{1-})-\partial_a^-\cH_t(\mu^{0+})\geq-\frac12\partial_t^-\cW_{t-}^2(\mu^0,\mu^1)\;.
\end{align}
\end{enumerate}
A time-dependent Markov triple $(\cX_t,Q_t,\pi_t)_t$ is called a
\emph{discrete super Ricci flow} if any of these equivalent properties hold.
\end{theorem}

The properties (I)--(IV) are natural discrete analogues of the
corresponding properties characterizing classical smooth super Ricci
flows discussed in Section \ref{sec:characterize}. An essential aspect
here is that the gradient estimate (II) and the transport estimate
(III) are requested to hold for \emph{all} $s\leq t$, i.e.~also
across singular times. This is what allows us to give a consistent
definition of discrete super Ricci flow through singularities where
the graph structure changes.

\smallskip Let us briefly comment on some ideas for the proof of
Theorem \ref{thm:srf-equiv-intro}. We will show that (I) implies (II)
via a classical interpolation argument, considering the quantity
$\Phi(r)=\Gamma_r(\hat P_{t,r}\mu,P_{r,s}\psi)$ and differentiating in
$r$. The crucial observation is that $\Phi$ is continuous at the
singular times. To show that (II) implies (III) we employ a dual
formulation of the discrete transport distance $\cW_t$ in the spirit
of the Kantorovich duality involving subsolutions to the
Hamilton-Jacobi equation, that was recently established independently
in \cite{EMW} and \cite{GLM17}, see Section \ref{sec:dual}. The reverse
implication will be shown by taking $\nu$ close to $\mu$ on the
geodesic in direction $\nabla\psi$ and employing again the duality.
The implication from (II)/(III) to (IV) constitutes the technical
core of the argument. Inspired by arguments in \cite{sturm2016}, we
will show that (II) implies that the heat flow can be characterized as
the gradient flow of the entropy w.r.t.~$\cW_t$ in the sense of a
dynamic evolution variational inequality. This together with (III)
will imply the dynamic convexity (IV). Finally, (IV) will imply (I)
after noting that $\Gamma_{2,t}(\mu,\cdot)$ coincides with the
Hessian of the entropy $\cH_t$ at $\mu$.

\smallskip

Our third main result concerns the consistency of discrete super Ricci
flow with classical super Ricci flows and more generally with the
synthetic definition of Sturm \cite{sturm2015} for time-dependent
metric measure spaces. In Theorem \ref{thm:stable} we identify a
suitable notion of convergence of a sequence
$(\cX^{(n)},Q^{(n)}_t,\pi^{(n)}_t)_{t\in I}$ of time-dependent Markov
triples to a time-dependent Riemannian manifold or metric measure
space $(X,d_t,m_t)_{t\in I}$ such that a limit of discrete super Ricci
flows is again a super Ricci flow. To this end, we assume that maps
$i_n:\cP(\cX^{(n)})\to\cP(X)$ exist and postulate a suitable sort of
$\Gamma$-convergence of the entropies and the transport distances
along these maps. Under some uniform regularity assumption on the
time-dependence this will suffice for the stability of super Ricci
flows. In order to pass to the limit we will employ an integrated
formulation of the dynamic convexity property (IV) already used in
\cite{sturm2015}.

We think of the Markov triples as finer and finer discrete
approximations of the spaces $(X,d_t,m_t)$. This approximation should
come with a natural way of extending (and regularizing) measures from
the discrete approximation to the full space, given by the maps
$i_n$. The purpose of our result is to identify sufficient conditions
for the stability of super Ricci flows. In practice, the
$\Gamma$-convergence of entropies should be a soft
requirement. Convergence of the transport distances seems harder to
establish. Some results are available in the static case by Gigli and
Mass \cite{GM12} and Trillos \cite{Tri} for a lattice resp.~point
cloud approximations of the torus, or by Gladbach, Maas and the second
author for finite element approximation of Euclidean domains
\cite{GKM}. Convergence results for discrete transport distances on
curved spaces remain an interesting open problem at the moment.

\subsection{Connection to the literature}
\label{sec:lit}
Let us briefly mention other related results in the literature. As
already discussed, our approach is close in spirit to the synthetic
approach to super Ricci flow in \cite{sturm2015,sturm2016} and
consistent with the discrete notion of Ricci curvature bounds
considered in \cite{EM12}. Many other approaches to Ricci curvature
for (weighted) graphs have been proposed, let us mention in particular
the combinatorial notion of Forman \cite{For03}, the coarse Ricci
curvature by Ollivier \cite{Oll09} based on the $L^1$ Kantorovich
distance, and approaches based on (modifications) of a discrete
Bakry--\'Emery $\Gamma_2$ criterion, see
e.g.~\cite{LLY11,BHLLMY15,DKZ}. A notion of discrete Ricci flow based
on Forman's combinatorial Ricci curvature has been studied recently in
\cite{WSJ17} and applied to the analysis of complex networks. Also the
latter two curvature notions could be used to define a notion of
(super) Ricci flow for weighted graphs. For Ollivier's curvature this
was proposed e.g.~in \cite{Tan15} motivated by the analysis of complex
cancer networks. We are not aware of any works studying these notions
in more detail. A lot of activity has been devoted to discrete notions
of Ricci flow in the more specific setting of triangulated
surfaces. See for instance the work of Chow and Luo \cite{CL03}
related to circle packings. These notions have broad applications in
graphics and medical imaging, for instance, see
e.g.~\cite{Zha15,Zen13}. A generalization of discrete Ricci flow to
higher dimensional simplicial structures termed simplicial Ricci flow
has recently been proposed in \cite{YauAL14a}. Also other curvature
flows such as the Yamabe flow have been considered in the discrete
setting \cite{Glick05}.

An advantage of our approach is that our notion of super Ricci flow
can naturally be defined through singularities. To our knowledge, this
has not been considered for the other notions discussed above. Another
advantage is that our discrete super Ricci flow yields strong control
on the heat flow on the evolving graph.

\subsection*{Organization}
The article is organized as follows. In Section \ref{sec:prelim} we
recall the notion of discrete transport distance on weighted graphs,
in particular its dual formulation that will be crucial in proving the
equivalence of the different characterizations of super Ricci
flows. We also recall the notion of entropic Ricci bounds for weighted
graphs. In Section \ref{sec:heat} we describe in detail the setting of
singular time-dependent Markov triples with changing graph structure that
we consider. Then we prove existence and uniqueness of solutions to
the heat equation and dual heat equation in this general
setting. In Section \ref{sec:RFequiv} we prove equivalence of the
different characterizations of super Ricci flows. Several examples of
super Ricci flows are presented in Section \ref{sec:ex}. Finally, in
Section \ref{sec:stability}, we discuss the consistency of our
discrete notion of super Ricci flow with the notion of super Ricci
flow for smooth manifolds or continuous metric measure spaces.

\subsection*{Acknowledgments}
\label{Acknowledgements}
M.E.~and E.K.~gratefully acknowledge support by the German Research
Foundation through the Hausdorff Center for Mathematics and the
Collaborative Research Center 1060 \emph{Mathematics of Emergent
  Effects}.

\section{Preliminaries on discrete optimal transport}
\label{sec:prelim}

Here we briefly recall the definitions of the discrete transport
distance $\cW$ and the associated Riemannian structure introduced
independently in \cite{Ma11,Mie11a}, and the entropic Ricci curvature
bounds for finite Markov chains introduced and studied in
\cite{EM12}. Finally we derive a dual formulation of the transport
distance.

\subsection{Discrete transport distance and Ricci bounds}
\label{sec:def}

Let $\cX$ be a finite set and let $Q:\cX\times\cX\to\R_+$ be a
collection of transition rates. The operator $\Delta$ acting on
functions $\psi : \cX \to \R$ via
\begin{align*}
\Delta \psi(x) = \sum_{y \in \cX} Q(x,y) \big(\psi(y) - \psi(x)\big)  
\end{align*}
is the generator of a continuous time Markov chain on $\cX$. We make
the convention that $Q(x,x)=-\sum_{y\neq x}Q(x,y)$ for all
$x\in\cX$. We assume that $Q$ is irreducible, i.e.~for all $x,y\in\cX$
there exists a path $x_0=x,x_1,\dots,x_n=y$ such that
$Q(x_i,x_{i+1})>0$. We assume moreover, that $Q$ is reversible. More precisely, we assume
that there exists a strictly positive probability measure $\pi$ on
$\cX$ such that the detailed-balance condition holds:
\begin{align}\label{eq:reversibility}
  Q(x,y)\pi(x) = Q(y,x)\pi(y)\quad \forall x,y\in\cX\;.
\end{align}
A triple $(\cX,Q,\pi)$ as above will be called a \emph{Markov triple}. 

We consider a distance $\cW$ on the set $\PX$ of probability measures
on $\cX$ defined as follows: for $\mu_0,\mu_1 \in \PX$ set
\begin{align}\label{eq:def-W}
  \cW(\mu_0,\mu_1)^2~&=~\inf_{\mu,
    V}\left\{\int_0^1\cA(\mu_t,V_t)\dd
    t~:~(\mu,V)\in\CE_1(\mu_0,\mu_1)\right\}\;,
\end{align}

where $\CE_T(\mu_0,\mu_1)$ denotes the collection of pairs
$(\mu,V)$ satisfying the continuity equation, more precisely, the following conditions:
\begin{align} \label{eq:conditions-2} 
 \left\{ \begin{array}{ll}
{(i)} & \mu : [0,T] \to \R^\cX  \text{ is continuous} \;;\\
{(ii)} &  \mu(0) = \mu_0\;, \qquad \mu(T) = \mu_1\;; \\
{(iii)} & \mu(t)\in\PX \text{ for all }t\in[0,T]\;;\\
{(iv)} & V  : [0,T] \to \R^{\cX\times\cX} \text{ is locally integrable}\;;\\
{(v)} & \text{For all $x \in \cX$ we have in the sense of distributions}\\
       &\displaystyle{\dot \mu_t(x) 
         + \frac12 \sum_{y \in \cX}
         \big(V_t(x,y) - V_t(y,x)\big)  = 0}\;.\
\end{array} \right.
\end{align}
The action $\cA$ is defined via
\begin{align*}
  \cA(\mu,V)~=~\frac12\sum\limits_{x,y}\frac{V(x,y)^2}{\Lambda(\mu)(x,y)}\;, \quad \Lambda(\mu)(x,y)=\hat\mu(x,y):=\Lambda\big(\mu(x)Q(x,y),\mu(y)Q(y,x)\big)\;,
\end{align*}
where $\Lambda$ denotes the logarithmic mean given by
\begin{align*}
  \theta(s,t) = \int_0^1s^\alpha t^{1-\alpha}\dd\alpha\;.
\end{align*}
More precisely, we set 
\begin{align*}
  \cA(\mu,V)~=~\frac12\sum\limits_{x,y}\alpha\Big(V(x,y),\mu(x)Q(x,y), \mu(y)Q(y,x)\Big)\;,
\end{align*}
where the convex and lower semicontinuous function
$\alpha : \R \times\R_+^2 \to \R \cup \{ + \infty\}$ is defined by
\begin{align*}
  \alpha(x,s,t) = \begin{cases}
\frac{x^2}{\Lambda(s,t)}\;,
                             & s,t \neq 0\;,\\                     
0\;,
                             &  \Lambda(s,t) = 0\text{ and } x = 0 \;,\\                             
                             + \infty\;,
                             & \text{else}\;.\end{cases}
\end{align*}
It is readily checked, that this formulation of $\cW$ is equivalent to
the one given in \cite[Lem.~2.9]{EM12}, in particular, in the the
definition of $\cW$ one can restrict the infimum to curves $\mu$ and
$V$ that are smooth.

It has been shown in \cite{Ma11} that $\cW$
defines a distance on $\PX$. It turns out that it is induced by a
Riemannian structure on the interior $\PXs$, consisting of all
strictly positive probability measures. The distance $\cW$ can be seen
as a discrete analogue of the Benamou--Brenier formulation \cite{BB00}
of the continuous $L^2$-transportation cost. The role of the
logarithmic mean is due to provide a discrete chain rule for the
logarithm, namely $\hat\rho\nabla\log\rho=\nabla\rho$, where we write
$\nabla\psi(x,y)=\psi(y)-\psi(x)$ and
$\hat\rho(x,y)=\Lambda\big(\rho(x),\rho(y)\big)$. The distance $\cW$
is tailor-made in this way such that the discrete heat equation
$\partial_t\rho=\Delta\rho$ is the gradient flow of the relative
entropy
\begin{align*}
 \cH(\mu) = \sum_{x \in \cX}  \log\frac{\mu(x)}{\pi(x)}\mu(x)
\end{align*}
w.r.t.~the Riemannian structure induced by $\cW$
\cite{Ma11,Mie11a}, making $\cW$ a natural replacement of the
$L^2$-Kantorovich distance  in the discrete setting. 

Every pair of measures $\mu_0, \mu_1 \in \PX$ can be joined by a constant speed
$\cW$-geodesic $(\mu_s)_{s\in[0,1]}$. Here constant speed geodesic
means that $\cW(\mu_s,\mu_t)=|s-t|\cW(\mu_0,\mu_1)$ for all
$s,t\in[0,1]$. The geodesic is a minimizer in \eqref{eq:def-W}.

In analogy with the approach of Lott--Sturm--Villani, the following
definition of a Ricci curvature lower bound has been given in \cite{EM12}.

\begin{definition}\label{def:intro-Ricci}
  $(\cX,Q,\pi)$ has \emph{Ricci curvature bounded from below by
    $\kappa \in \R$}, if for any constant speed geodesic $\{\mu_t\}_{t
    \in [0,1]}$ in $(\PX, \cW)$, we have
  \begin{align*}
    \cH(\mu_t) \leq (1-t) \cH(\mu_0) + t \cH(\mu_1) -
    \frac{\kappa}{2} t(1-t) \cW(\mu_0, \mu_1)^2\;.
  \end{align*}
  In this case, we write $\Ric(\cX,Q,\pi) \geq \kappa$.
\end{definition}

\subsection{Riemannian structure and Bochner-type inequality}
\label{sec:BA}

Entropic curvature bounds can be expressed equivalently via an
inequality resembling Bochner's inequality in Riemannian geometry. To this end, let us describe the Riemannian structure induced by
$\cW$.

To alleviate notation, let us denote for $\Phi,\Psi\in\R^{\cX\times\cX}$ their Euclidean inner product by
\begin{align*}
  \ip{\Phi,\Psi}=\frac12\sum_{x,y\in\cX}\Phi(x,y)\Psi(x,y)\;.
\end{align*}
At each $\mu\in\cP_*(\cX)$ the tangent space to $\cP_*(\cX)$ is given
by $\mathcal{T}=\{s\in\R^\cX:\sum_xs(x)=0\}$. Given $\psi\in\R^\cX$ we
denote by $\nabla\psi\in\R^{\cX\times\cX}$ the discrete gradient of
$\psi$, i.e.~the quantity $\nabla\psi(x,y)=\psi(y)-\psi(x)$. Let
$\mathcal{G}=\{\nabla\psi:\psi\in\R^\cX\}$ denote the set
of all discrete gradient fields and note that $\mathcal{G}$ is in bijection to the set $\mathcal{G'}=\{\psi\in\R^{\cX}:\psi(x_0)=0\}$. In
\cite[Sec.~3]{Ma11} it has been shown that for each $\mu\in\cP_*(\cX)$, the map
\begin{align*}
  K_\mu: \psi \mapsto \sum_y\nabla\psi(y,x)\Lambda(\mu)(x,y)\,,
\end{align*}
defines a linear bijection between $\mathcal{G}$ and the tangent space
$\mathcal{T}$. This identification can be used to define a Riemannian metric tensor on
$\cP_*(\cX)$ by introducing the scalar
product $\langle\cdot,\cdot\rangle_\mu$ on $\mathcal{G}$ given by
\begin{align*}
  \langle\nabla\psi,\nabla\phi\rangle_\mu=\ip{\nabla\psi,\nabla\phi\cdot\Lambda(\mu)}= \frac12\sum_{x,y}\nabla\psi(x,y)\nabla\phi(x,y)\Lambda(\mu)(x,y)\;.
\end{align*}
Then $\cW$ is the Riemannian distance associated to this Riemannian
structure. Note that if we introduce the divergence of $\Phi\in\R^{\cX\times\cX}$ via 
\begin{align*}
  \nabla\cdot\Phi(x):=\frac12\sum_{y\in\cX}\Phi(x,y)-\Phi(y,x)\;,
\end{align*}
we can write for short $K_\mu\psi=\nabla\cdot\big(\Lambda(\mu)\cdot\psi\big)$.

Let us introduce the following \emph{integrated carr\'e du champs operators}. For $\mu\in\PX$ (resp.~$\mu\in\cP_*(\cX)$) and $\psi\in\R^\cX$ set
\begin{align}\label{eq:Gamma}
  \Gamma(\mu,\psi)&:=\ip{\nabla\psi,\nabla\psi\cdot\Lambda(\mu)} = \norm{\nabla\psi}_\mu^2\;,\\\label{eq:Gamma2}
  \Gamma_2(\mu,\psi)&:=\frac12\ip{\nabla\psi,\nabla\psi\cdot\hat\Delta \Lambda(\mu)} -\ip{\nabla\psi,\nabla\Delta\psi\cdot\Lambda(\mu)}\;,
\end{align}
where we have used the notation
\begin{align*}
 \hat \Delta \Lambda(\mu)(x,y) ~&:=~  
 \Big[\partial_1\Lambda\big(\rho(x),\rho(y)\big) \Delta\rho(x) 
  + \partial_2\Lambda\big(\rho(x),\rho(y)\big) \Delta\rho(y)\Big]Q(x,y)\pi(x)\;,
\end{align*}
with $\rho(x)=\mu(x)/\pi(x)$ and where the multiplication with
$\Lambda(\mu)$ and $\hat\Delta\Lambda(\mu)$ is defined component wise.

Entropic Ricci bounds, i.e.~convexity of the entropy along
$\cW$-geodesics, are determined by bounds on the Hessian of the
entropy $\cH$ in the Riemannian structure defined above. An explicit
expression of the Hessian at $\mu\in\cP_*(\cX)$ is given by
\begin{align*}
  \Hess\cH(\mu)[\nabla\psi] = \Gamma_2(\mu,\psi)\;.
\end{align*}
We then have the following equivalent characterization of entropic
Ricci bounds.

\begin{proposition}[{\cite[Thm.~4.4]{EM12}}]\label{prop:Ric-equiv}
  A Markov triple $(\cX,Q,\pi)$ satisfies $\Ric(\cX,Q,\pi)\geq\kappa$
  if and only if for every $\mu\in\cP_*(\cX)$ and every
  $\psi\in\R^\cX$ we have
  \begin{align*}
    \Gamma_2(\mu,\psi)~\geq~\kappa \Gamma(\mu,\psi)\;.
  \end{align*}
\end{proposition}

Note that this statement is non-trivial since the Riemannian metric
degenerates at the boundary of $\PX$. In view of \eqref{eq:Gamma},
\eqref{eq:Gamma2}, the criterion above closely resembles (an
integrated version of) the classical Bochner inequality or
Bakry--\'Emery $\Gamma_2$-criterion. Namely, a Riemannian manifold $M$
satisfies $\Ric\geq\kappa$ if and only if for every smooth
function $\psi:M\to\R$ and probability $\mu=\rho\mathrm{vol}$ we have:
\begin{align*}
\int_M \frac12 \left[\Delta\rho |\nabla\psi|^2 -\rho\langle\nabla\psi,\nabla \Delta\psi\rangle \right] \dd\mathrm{vol}
~\geq~ \int_M\rho|\nabla\psi|^2\dd\mathrm{vol}\;,  
\end{align*}
where $\nabla$ now denotes the usual gradient and $\Delta$ denotes the
Laplace--Beltrami operator. In fact, the left hand side equals the
Hessian of the entropy in Otto's formal Riemannian structure on
$\cP(M)$ associated with the $L^2$-Kantorovich distance $W_2$.
\eqref{eq:Gamma} and \eqref{eq:Gamma2} should be seen as discrete
analogues of the integrated carr\'e du champs operators
$\int\Gamma(\psi)\dd\mu$ and $\int\Gamma_2(\psi)\dd\mu$ appearing in
the right resp.~left hand side of Bochner's inequality.

\subsection{Duality for discrete optimal transport}
\label{sec:dual}

Here, we recall a dual formulation for the discrete transport distance
that has been established in \cite{EMW} and which can be seen as a
discrete analogue of the Kantorovich duality. A very similar result in
a slightly more restrictive setting has been proven in \cite{GLM17}
and also existence of dual optimizers has been established, see
Prop.~3.10 and Thm.~5.10, 7.4 there.

\begin{definition}[Hamilton-Jacobi subsolution]\label{def:HJ}
We say that a function $\varphi\in H^1\big((0,T);\R^\cX\big)$ is a \emph{Hamilton--Jacobi subsolution} if for a.e. $t$ in $(0,T)$ we have
\begin{align}\label{eq:HJ}
 \ip{\dot\varphi_t,\mu} + \frac12\norm{\nabla\varphi_t}^2_\mu \leq 0 \quad\forall \mu\in\cP(\cX)\;.  
\end{align}
The set of all Ha\-mil\-ton--Jacobi subsolutions is denoted $\HJ^T_\cX$.
\end{definition}

\begin{remark}\label{rem:HJ-scaling}
 Given $\varphi\in\HJ_\cX^T$ and $\lambda>0$, set $\varphi^\lambda_t:=\lambda \varphi_{\lambda t}$. Then $\varphi^\lambda\in \HJ_\cX^{\lambda T}$.
\end{remark}

\begin{theorem}[{Duality formula, \cite[Thm.~3.3]{EMW}}]\label{thm:dual}
For $\mu_0,\mu_1\in\PX$ we have
\begin{align}\label{eq:dual-W}
 \frac12\cW^2(\mu_0,\mu_1)=\sup\big\lbrace
\ip{\varphi_1,\mu_1}-\ip{\varphi_0,\mu_0}\,:\, \varphi\in \HJ^1_\cX\big\rbrace.
\end{align}
This representation remains true if the supremum is restricted to
functions $\phi\in C^1\big([0,1],\R^\cX\big)$ satisfying
\eqref{eq:HJ}.
\end{theorem}

For the readers convenience, let us also recall the heuristic
derivation of the duality result above. We start by introducing a
Lagrange multiplier for the continuity equation constraint and write
\begin{align}\label{eq:dual1}
  \frac12\cW(\mu_0,\mu_1)^2~&=~\inf\limits_{\mu,V}\sup\limits_\phi\left\{\int_0^1\frac12\cA(\mu_t,V_t)\dd
    t + \int_0^1\ip{\phi_t,\dot\mu_t+\nabla\cdot V_t} \dd
    t\right\}\;,
\end{align}
where the supremum is taken over all (sufficiently smooth) functions
$\phi:[0,1]\to\R^\cX$ and the infimum is taken over all (sufficiently
smooth) curves $\mu: [0,1] \to \R_+$ connecting $\mu_0$ and
$\mu_1$, and over all $V : [0,1] \to \R^{\cX \times \cX}$. Here we do
not require that $(\mu,V)$ satisfies the continuity equation, but the
inner supremum takes the value $+\infty$ if $(\mu,V)$ does not belong
to $\CE_1(\mu_0,\mu_1)$. We also do not require that $\mu$ takes
values in $\PX$, but this is automatically enforced by the continuity
equation. Continuing \eqref{eq:dual1} we obtain via integration by
parts
\begin{align*}
  \frac12\cW(\mu_0,\mu_1)^2~=~\inf\limits_{\mu,V}\sup\limits_\phi\left\{\ip{\phi_1,\mu_1}-\ip{\phi_0,\mu_0}
    + \int_0^1\frac12\cA(\mu_t,V_t)
    -\ip{\dot\phi_t,\mu_t}-\ip{\nabla\phi_t, V_t} \dd
    t\right\}\;.
\end{align*}
Applying the min--max principle and calculating the infimum we obtain
\begin{align*}
  \frac12\cW(\mu_0,\mu_1)^2~=~\sup\big\{\ip{\phi_1,\mu_1}-\ip{\phi_0,\mu_0}
    ~:~\phi\in \cH\big\}\;,
\end{align*}
where $\cH$ is the set of $\phi$ such that for a.e.~$t$ and all $\mu$ and $V$
\begin{align*}
  \frac12\cA(\mu,V) -\ip{\dot\phi_t,\mu}-\ip{\nabla\phi_t,V}\geq 0\;.
\end{align*}
This is due to the fact that the quantity to be
minimized is positive $1$-homogeneous in $(\mu, V)$, hence the
infimum takes the value $-\infty$ if $\phi$ does not belong to
$\cH$. The last inequality rewrites as
\begin{align*}
  0~&\leq~ \frac12\cA(\mu,V) -\ip{\dot\phi_t,\mu} - \ip{\nabla\phi_t,V}\\
  &=~ \frac14\sum_{x,y}\bigg[\frac{V(x,y)^2}{\hat\mu(x,y)}-2\nabla\phi_t(x,y)V(x,y)\bigg] -\ip{\dot\phi_t,\mu}\\
  &=~ \frac14\sum_{x,y}\left(\frac{1}{\hat\mu(x,y)}\Big[V(x,y)-\nabla\phi_t(x,y)\hat\mu(x,y)\Big]^2
 - \abs{\nabla\phi_t(x,y)}^2\hat\mu(x,y)\right)\\
  &\qquad -\ip{\dot\phi_t,\mu}\;.
\end{align*}
Minimizing over $V$ we conclude that $\phi \in \cH$ iff the
inequality
\begin{align*}
  \ip{\dot\phi_t,\mu}+\frac12\norm{\nabla\phi_t}^2_\mu\leq 0\;,
\end{align*}
holds for all $\mu\in\R^\cX_+$, i.e.~iff $\phi\in\HJ_\cX$.

\section{Heat equations on time-dependent Markov triples}
\label{sec:heat}
In this section, we study the heat equation on a time-dependent Markov
triple. This will be a crucial tool for the characterization of super
Ricci flows in Section \ref{sec:RFequiv}. We will first describe in
Section \ref{sec:topo} the setting of time-dependent Markov chains
that we consider, where the state space is allowed to vary and may
feature collapse or creation of vertices. We will briefly discuss in
Section \ref{sec:heat-simple} the heat equation associated to a time
inhomogeneous Markov chain on a fixed state space. In Section
\ref{sec:heat-sing} we will give existence and uniqueness results for
the heat equation and the adjoint heat equation on measures in the
general singular space-time setting.

\subsection{Singular discrete space-times}
\label{sec:topo}

We consider a time dependent family of Markov triples
$(\cX_t,Q_t,\pi_t)_{t\in[0,T]}$. Recall that this means that for each
$t\in[0,T]$, $\cX_t$ is a finite set, $Q_t$ is the matrix of
transition rates $\big(Q_t(x,y)\big)_{x,y\in\cX_t}$ with
$Q_t(x,y)\geq0$ for $x\neq y$ and $Q_t(x,x)=-\sum_{y\neq x}Q_t(x,y)$,
and $\pi_t$ is a strictly positive probability measure on $\cX_t$ such
that $Q_t$ is reversible w.r.t.~$\pi_t$.

\begin{definition}\label{def:sing-triple}
  A \emph{singular time-dependent Markov triple} is a family
  $(\cX_t,Q_t,\pi_t)_{t\in[0,T]}$ of Markov triples such that there
  exist a partition $0=t_0<t_1<\dots<t_n=T$, finite sets
  $\bar\cX_0,\ldots,\bar\cX_n$ and $\cX_0,\ldots\cX_{n-1}$, and
  surjective maps $s_i:\cX_i\to\bar\cX_i$ and
  $c_i:\cX_i\to\bar\cX_{i+1}$ such that the following conditions hold:
\begin{itemize}
\item[(1)] $\cX_{t_i}=\bar\cX_i$ and $\cX_t=\cX_i$ for
  $t\in I_i:=(t_i,t_{i+1})$ for all $i=0,\ldots,n-1$;
\item[(2)] $t\mapsto \pi_t(x)$ is Lipschitz on $I_i$ for all $i$ and $x\in\cX_i$ and the limits
  \begin{align*}
    \pi_i^c(x):=\lim_{t\uparrow t_{i+1}}\pi_t(x)\;,\quad    \pi_i^s(x):=\lim_{t\downarrow t_{i}}\pi_t(x)
  \end{align*}
 exist in $(0,1)$;

\item[(3)] $t\mapsto Q_t(x,y)$ is locally log-Lipschitz on $I_i$,
  i.e.~for each $x\neq y$ either $Q_t(x,y)=0$ for all $t\in I_i$ or
  $Q_t(x,y)>0$ for all $t\in I_i$ and the map $t\mapsto \log Q_t(x,y)$
  is locally Lipschitz and the limits
  \begin{align*}
    Q_i^c(x,y):=\lim_{t\uparrow t_{i+1}}Q_t(x,y)\;,\quad    Q_i^s(x,y):=\lim_{t\downarrow t_{i}}Q_t(x,y)
  \end{align*}
 exist in $[0,\infty]$. In case $Q_i^c(x,y)=+\infty$ resp. $Q_i^s(x,y)=+\infty$, we assume further that
 \begin{align}\label{eq:rate-int-infty}
   \int^{t_{i+1}}Q_t(x,y)\dd t = +\infty\;,\quad\text{resp.~}\int_{t_{i}}Q_t(x,y)\dd t = +\infty\;;
 \end{align}
\item[(4)] we have that $c_i(x)=c_i(y)=z\in\bar\cX_{i+1}$ iff
  $x\overset{c}{\leftrightarrow}y$ and $s_i(x)=s_i(y)=z\in\bar\cX_{i}$
  iff $x\overset{s}{\leftrightarrow}y$, where we write
  $x\overset{c}{\leftrightarrow}y$ iff there exists a path
  $x=x_1,x_2,\dots,x_n=y$ with $Q^c_i(x_j,x_{j+1})=+\infty$ for
  $j=0,n-1$ and similarly for $x\overset{s}{\leftrightarrow}y$ (note
  that these define equivalence relations on $\cX_i$ by detailed
  balance and $c_i$, $s_i$ are the associated quotient maps);

\item[(5)] we have that for $z\in\bar\cX_i$
  \begin{align}\label{eq:limit-pi}
    \pi_{t_i}(z) = \sum_{x\in s_i^{-1}(z)} \pi_i^s(x) =  \sum_{x\in c_{i-1}^{-1}(z)} \pi_{i-1}^c(x)\;,
  \end{align}
    and that for $z,z'\in\bar\cX_i$
    \begin{align}\label{eq:limit-Q}
      Q_{t_i}(z,z') &= \frac{1}{\pi_{t_i}(z)}\sum_{x\in s_i^{-1}(z),x'\in s_i^{-1}(z')}Q_i^s(x,x')\pi_i^s(x)\\\nonumber
&= \frac{1}{\pi_{t_i}(z)}\sum_{x\in c_{i-1}^{-1}(z),x'\in c_{i-1}^{-1}(z')}Q_{i-1}^c(x,x')\pi_{i-1}^c(x)\;.
    \end{align}
\end{itemize}
\end{definition}
Note that \eqref{eq:limit-pi}, \eqref{eq:limit-Q} are automatically
consistent with the requirement that $Q_{t_i}$ and $\pi_{t_i}$ satisfy
the detailed balance condition.

The interpretation of these assumption is the following. During the
open intervals $I_i=(t_i,t_{i+1})$ the graph structure does not change
and we have a log-Lipschitz control on the rates. At the times $t_i$
the topology of the graph can change and (a combination of) the
following event(s) can occur:
\begin{itemize}
\item vertices can disconnect, i.e.~$Q_t(x,y)\searrow0$ as
  $t\uparrow t_i$ or start to connect, i.e.  $Q_t(x,y)\searrow0$ as
  $t\downarrow t_i$,
\item a group of vertices can collapse to a point, here $c_i^{-1}(z)$
  is the set of vertices of $\cX_i$ that collapse to
  $z\in\bar\cX_{i+1}$, this happens iff each pair of vertices in the
  group is connected via a path of edges whose weights explode,
\item a point can spawn a group of new vertices at later times (same
  as collapse but backwards in time), here $s_i^{-1}(z)$ is the set of
  vertices of $\cX_i$ that are spawned by $z\in\bar\cX_{i}$,
\item collapsing happens in a controlled way, more precisely, ratios
  of rates inside a collapsing group have a limit and
  \begin{align}\label{eq:local-eq}
    \bar\pi^{c,z}_i(x):=\pi_i^c(x)\left(\sum_{y\in c_i^{-1}(z)}\pi^c_i(y)\right)^{-1}
  \end{align}
  can be seen as the asymptotic equilibrium measure on $c_i^{-1}(z)$
  as we ``zoom into the collapse''; similarly for spawning.
\end{itemize}

\begin{figure}[h]
  \centering
  \resizebox{4in}{!}{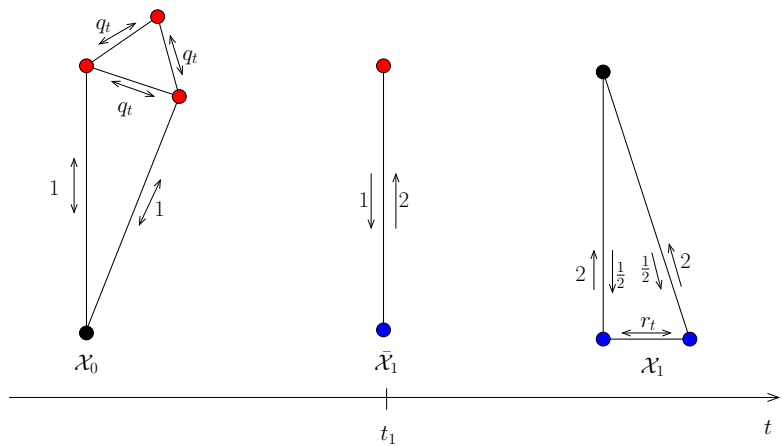}
  \caption{A singular time-dependent Markov chain}
  \label{fig:example}
\end{figure}

\begin{example}\label{ex:toy}
  A simple example of a singular time-dependent Markov triple
  satisfying these conditions is given in Figure
  \ref{fig:example}. Here the transition rate $Q_t(x,y)$ is depicted
  with arrows from $x$ to $y$ along the edges. The three red vertices
  collapse at time $t_1$ to a single vertex. Afterwards, the blue
  vertex of $\bar\cX_1$ spawns a new vertex. Here we could set for
  instance $q_t=1/(t_1-t)$ and $r_t= 1/(t-t_1)$ so that
  $\int^{t_1}q_t\dd t=\int_{t_1}r_t\dd t =+\infty$.
\end{example}

In Section \ref{sec:ex} we discuss more examples that arise as super
Ricci flows and which feature a similar $1/t$ behavior of the rates
at singular times.

\begin{remark}\label{rem:explosion}
  From the point of view of the heat equation on the time dependent
  graph, one should expect that if a group of vertices collapses at
  $t_i$, then the solution has already equilibrated on these vertices
  before the collapsing time. This is the case if the average number
  of jumps between these vertices before $t_i$ is infinite. This is
  the reason why we assume \eqref{eq:rate-int-infty}.

  Without this condition the heat equation does not
  necessarily equilibrate on vertices with exploding rates before a
  singular time. Consider e.g.~the two-point space $\cX_0=\{a,b\}$ with
$Q_t(a,b)=Q_t(b,a)=q_t$ on $(t_0,t_1)$. Then, for the solution 
to the heat equation $\partial_t\psi=\Delta_t\psi$ (see below) we have explicitly with $\delta_t=\psi(t,b)-\psi(t,a)$ that 
\begin{align*}
 \frac{\dd}{\dd t} \delta_t = -\delta_tq_t\;,\quad \delta_t = \delta_{s}\exp\Big(-\int_s^tq_r\dd r\Big)\;.
\end{align*}
Choosing for instance $q_t=1/\sqrt{t_1-t}$, we see that $\delta_t$ does not vanish as $t\uparrow t_1$ unless $\delta_s=0$. 
\end{remark}

We will denote by $\dot Q_t,\dot\pi_t$ the derivatives w.r.t.~$t$ of
$Q_t$ and $\pi_t$ which exist for a.e.~$t\in[0,T]$ by the assumption
of Lipschitz continuity, (2) and (3) above.

We denote by $\Delta_t$ the Laplace operator associated to $Q_t$
acting on function $\psi\in\R^\cX$ via
\begin{align*}
  \Delta_t\psi (x) = \sum_{y\in\cX} \psi(y) Q_t(x,y) = \sum_{y\in\cX} \big[\psi(y)-\psi(x)\big] Q_t(x,y)\;. 
\end{align*}
Let us introduce the inner products on $\R^{\cX_t}$ and $\R^{\cX_t\times\cX_t}$ respectively given by 
\begin{align*}
  \ip{\psi,\phi}_{\pi_t}:=\sum_{x\in\cX_t} \psi(x)\phi(x)\pi_t(x)\;,\quad\ip{\Psi,\Phi}_{\pi_t}:=\frac12\sum_{x,y\in\cX_t} \Psi(x,y)\Phi(x)Q_t(x,y)\pi_t(x)\;.
\end{align*}
Then $\Delta_t$ is symmetric, i.e.~we have that
$\ip{\psi,\Delta_t\phi}_{\pi_t}=\ip{\Delta_t\psi,\phi}_{\pi_t}$.
Moreover, we have the following integration by parts relation
$\ip{\nabla\phi,\nabla\psi}_{\pi_t}=-\ip{\Delta_t\phi,\psi}_{\pi_t}$
for all $\phi,\psi\in\R^{\cX_t}$.

For a function $\sigma\in\R^\cX$ (viewed as a signed measure on $\cX$)
we define the adjoint Laplace operator $\hat\Delta_t$ via
\begin{align*}
  \hat\Delta_t\sigma(x) = \sum_{y\in\cX}Q_t(y,x)\sigma(y) = \sum_{y\neq x}Q_t(y,x)\sigma(y) -\sum_{y\neq x}Q_t(x,y)\sigma(x)\;.
\end{align*}
Note that if $\sigma=\rho\pi_t$ for some $\rho\in\R^\cX$ then
$\hat\Delta_t\sigma=(\Delta_t\rho)\pi_t$ by the detailed balance condition. 

Further, denoting the Euclidean inner product $\psi,\sigma\in\R^{\cX_t}$ (viewed as the integral of $\psi$ against $\sigma$) by
\begin{align*}
  \ip{\psi,\sigma}:=\sum_{x\in\cX_t}\psi(x)\sigma(x)\;,
\end{align*}
we have that $\ip{\Delta_t\psi,\sigma} = \ip{\psi,\hat\Delta_t\sigma}$.

\subsection{The heat equations}\label{sec:heat-simple}
Let us first consider the situation of a time-dependent Markov triple
$(\cX,Q_t,\pi_t)_{t\in(0,T)}$ with a fixed space $\cX$ and
time-dependent rates $Q_t$ and measure $\pi_t$ that are locally
log-Lipschitz in $t$.

Given $\bar\psi\in\R^\cX$ and $0<s<T$ we say that a function
$\psi:[s,T)\times\cX\to\R$ solves the time-dependent heat equation
with initial condition $\bar\psi$ if $t\mapsto\psi(t,x)$ is
differentiable on $(s,T)$ and continuous at $s$ for all $x$ and
\begin{align*}
 \partial_t\psi(t,x)&=\Delta_t\psi(t,x) \quad\text{on } (s,T)\times \cX\;,\\
\psi(s,\cdot)&=\bar\psi\;.
\end{align*}
Note that by continuity of $t\mapsto Q_t$,
there is a unique such solution. Thus we can define the heat
propagator $P_{t,s}:\R^\cX\to\R^\cX$ by setting
$P_{t,s}\bar\psi=\psi(t,\cdot)$, where $\psi$ is the above solution.

Given $\bar\sigma\in\R^\cX$ and $0<t<T$, we say that a function
$\sigma:(0,t]\to\R^\cX$ satisfies the adjoint heat equation for
measures with terminal condition $\bar\sigma$ if $s\mapsto\sigma(s,x)$
is differentiable on $(0,t)$ and continuous at $t$ for all $x$ and
\begin{align*}
 \partial_s\sigma(s,x)&=-\hat\Delta_s\sigma(s,x)\quad\text{on }(0,t)\times \cX\;,\\
\sigma(t,\cdot)&=\bar\sigma\;.
\end{align*}
There exist a unique such solution. We define the adjoint heat
propagator $\hat P_{t,s}:\R^\cX\to\R^\cX$ by setting
$\hat P_{t,s}\bar\sigma=\sigma(s,\cdot)$, where $\sigma$ is the above solution.

Note that if $\bar \sigma=\bar\rho\pi_t$, then we have
$\hat P_{t,s}\bar \sigma = \rho_s\pi_s$, where $\rho_s$ solves the adjoint heat equation
 \begin{align*}
 \partial_s\rho(s,x)=-\Delta_s\rho(s,x) -\dot p_s(x)\rho(s,x)\;,
\end{align*}
where $p_s=\log \pi_s$. Note that $P_{t,s}$ and $\hat P_{t,s}$ are
linear operators. Moreover, they are adjoint in the following sense:
for all $0<s\leq t<T$ and $\bar\sigma\in\R^\cX$ we have
\begin{align}\label{eq:adjoint-simple}
  \ip{P_{t,s} \bar\psi,\bar\sigma}=\ip{\bar\psi,\hat P_{t,s}\bar\sigma}\;.
\end{align}
Indeed, setting $\psi_r=P_{r,s}\bar\psi$ and $\sigma_r=\hat P_{t,r}\bar\sigma$ for $s\leq r\leq t$ we have 
\begin{align*}
  \frac{\dd}{\dd r} \ip{\psi_r,\sigma_r} 
=
\ip{\Delta_r\psi_r,\sigma_r} -\ip{\psi_r,\hat\Delta_r\sigma_r} =0\;.
\end{align*}

We note the following maximum and positivity principles for the (adjoint) heat equation.
\begin{lemma}\label{lem:max-principle}
  For all $0<s<t<T$ and $\bar\psi,\bar\sigma\in\R^\cX$ and $x\in\cX$ we have that
  \begin{align}\label{eq:max-principle}
   &\min_{y\in\cX} \bar\psi(y)\leq  P_{t,s}\bar\psi(x)
    \leq \max_{y\in\cX}\bar\psi (y)\;,\\\label{eq:max-principle-adj}
  &\min_{y\in\cX} \bar\sigma(y)\leq \hat P_{t,s}\bar\sigma(x)
    \leq \max_{y\in\cX}\bar\sigma (y)\;.
  \end{align}
  Moreover, $P_{t,s}\bar\psi$ and $\hat P_{t,s}\bar\sigma$ are
  strictly positive provided that $\bar\psi$ and $\bar\sigma$ are
  non-negative and not identically $0$.
\end{lemma}

\begin{proof}
  Let us first show that $P_{t,s}\bar\psi\geq0$ whenever
  $\bar\psi\geq0$. For this define
  $\psi_t^-:=\max\{-P_{t,s}\bar \psi,0\}$ and
  $\psi_t=P_{t,s}\bar\psi$. For $s<t<T$, $r\mapsto\log\pi_r$ is
  Lipschitz on $[s,t]$ with some constant $L$. Thus, we obtain
\begin{align*}
0&\leq\frac12\sum_{x,y\in\cX}(\psi_t^-(x)-\psi_t^-(y))^2Q_t(x,y)\pi_t(x)\\
&\leq -\frac12\sum_{x,y\in\cX}(\psi_t^-(x)-\psi_t^-(y))(\psi_t(x)-\psi_t(y))Q_t(x,y)\pi_t(x)\\
&=\sum_{x\in\cX}\psi_t^-(x)\Delta_t\psi_t(x)\pi_t(x)=\sum_{x\in\cX}\psi_t^-(x)\partial_t\psi_t(x)\pi_t(x)\\
&=-\frac12\sum_{x\in\cX}\partial_t(\psi_t^-(x))^2\pi_t(x)
\leq-\frac12e^{Lt}\partial_t\sum_{x\in\cX}e^{-Lt}(\psi_t^-(x))^2\pi_t(x),
\end{align*}
and in particular
\begin{align*}
0=\sum_{x\in\cX}e^{-Ls}(\psi_s^-(x))^2\pi_s(x)\geq \sum_{x\in\cX}e^{-Lt}(\psi_t^-(x))^2\pi_t(x)\;,
\end{align*}
which implies that $\psi_t\geq 0$.  

Now, let $m=\min_{y\in\cX}\bar\psi(y)$ and
$M=\max_{y\in\cX}\bar\psi(y)$. Then \eqref{eq:max-principle} follows
similarly by choosing $\psi_t^-:=\max\{-(P_{t,s}\bar \psi-m),0\}$ and
$\psi_t^+:=\max\{P_{t,s}\bar \psi-M,0\}$ respectively.

  To show \eqref{eq:max-principle-adj} it suffices to note that
  \begin{align*}
  \hat P_{t,s}\bar\sigma (x) = \ip{\delta_x,\hat P_{t,s}\bar\sigma} = \ip{P_{t,s} \delta_x,\bar\sigma} 
  \end{align*}
  and to apply \eqref{eq:max-principle}.

  The last statement follows from the fact that due to the Lipschitz
  assumption, the transition rates can be controlled on each compact
  subinterval of $(0,T)$ and then applying standard results for
  time-homogeneous Markov chains and the duality
  \eqref{eq:adjoint-simple}.
\end{proof}

In particular, we see that the heat equation preserves constants,
i.e.~$P_{t,s}\bar\psi\equiv c$ provided $\bar\psi\equiv c$. On the
other hand, the adjoint heat equation preserves mass, i.e.~
\begin{align}\label{eq:mass-preserv}
  \sum_{x\in\cX}\hat P_{t,s}\bar\sigma(x) = \sum_{x\in\cX}\bar\sigma(x)\;.
\end{align}
(this follows form \eqref{eq:adjoint-simple} choosing $\psi\equiv
1$).
Combining with the maximum principle, we see that the adjoint heat
equation preserves probability measures, i.e.
$\hat P_{t,s}\mu\in\cP(\cX)$ provided $\mu\in\cP(\cX)$.

Using the propagator, the (adjoint) heat equation reads
\begin{align*}
  \partial_tP_{t,s}\bar\psi = \Delta_tP_{t,s}\bar\psi\;,
  \quad
\partial_s\hat P_{t,s}\bar\sigma = -\Delta_s\hat P_{t,s}\bar\sigma\;.
\end{align*}
We can also take the derivative in the other time parameter, obtaining
\begin{align}\label{eq:other-index}
  \partial_sP_{t,s}\bar\psi = -P_{t,s}\Delta_s\bar\psi\;,
  \quad
\partial_s\hat P_{t,s}\bar\sigma = \hat P_{t,s}\Delta_t\bar\sigma\;.
\end{align}
This follows by noting that for $h>0$ we have
\begin{align*}
   P_{t,s+h}\bar\psi- P_{t,s}\bar\psi &= 
   P_{t,s+h}\Big[\bar\psi-P_{s+h,s}\bar\psi\Big]
= - P_{t,s+h} \int_s^{s+h}\Delta_r P_{r,s}\bar\psi\dd r\;,
\end{align*}
and then dividing by $h$ and letting $h\downarrow 0$. Similarly, one argues
for the left derivative and for the adjoint equation.

\subsection{The heat equations on singular space-times}
\label{sec:heat-sing}

Now, let us consider to the general setting of Section \ref{sec:topo}
and consider a singular time-dependent Markov triple
$(\cX_t,Q_t,\pi_t)_{t\in[0,T]}$ according to Definition
\ref{def:sing-triple}. We will show existence and uniqueness of
solutions to the heat equations on functions and measures across
singular times. To this end for $0\leq s<t\leq T$, let us define
space-time during the interval $[s,t]$ by setting
\begin{align}\label{eq:space-time}
  \cS_{s,t}:=\big\{(r,x) : r\in[s,t],~x\in\cX_r\big\}\;.
\end{align}

\begin{theorem}\label{thm:heat-eq}
  Given $s\in[0,T]$ and $\bar\psi\in\R^{\cX_s}$ there exist a unique
  function $\psi:\cS_{s,T}\to\R$ with the following properties:
  \begin{itemize}
  \item[(i)] $\psi(s,\cdot)=\bar\psi$,
  \item[(ii)] $t\mapsto \psi(t,\cdot)$ is differentiable on
    $I_i=(t_i,t_{i+1})$ and satisfies
    $\partial_t\psi(t,x)=\Delta_t\psi(t,x)$ on $I_i\times\cX_i$,
  \item[(iii)] for all $z\in \bar\cX_i$, $x\in s_i^{-1}(z)$ and $y\in c_{i-1}^{-1}(z)$ we have
    \begin{align}\label{eq:heat-eq-cont}
      \psi(t_i,z)=\lim_{t\downarrow t_i}\psi(t,x)
=\lim_{t\uparrow t_i}\psi(t,y)\;.
 \end{align}
\end{itemize}
Given $t\in[0,T]$ and $\bar\sigma\in\R^{\cX_t}$ there exist a unique
function $\sigma:\cS_{0,t}\to[0,\infty)$ with the following
properties:
  \begin{itemize}
  \item[(i)] $\sigma(t,\cdot)=\bar\sigma$,
  \item[(ii)] $s\mapsto \sigma(s,\cdot)$ is differentiable on
    $I_i=(t_i,t_{i+1})$ and satisfies
    $\partial_s\sigma(s,x)=-\hat\Delta_s\sigma(s,x)$ on
    $I_i\times\cX_i$,
  \item[(iii)] for all $z\in \bar\cX_i$ we have 
    \begin{align}\label{eq:heat-eq-adj-cont}
     \sigma(t_i,z)= \sum_{x\in s_i^{-1}(z)}\lim_{s\downarrow t_i}\sigma(s,x)
 =\sum_{y\in c_{i-1}^{-1}(z)}\lim_{s\uparrow t_i}\sigma(s,y) \;.
    \end{align}
  \end{itemize}
\end{theorem}

We define the heat propagator $P_{t,s}:\R^{\cX_s}\to\R^{\cX_t}$ and
adjoint heat propagator $\hat P_{t,s}:\R^{\cX_t}\to\R^{\cX_s}$ by setting
\begin{align*}
  P_{t,s}\bar\psi=\psi(t,\cdot)\;,\quad \hat P_{t,s}\bar\sigma = \sigma(s,\cdot)\;,
\end{align*}
where $\psi$ and $\sigma$ are the solutions given by the previous
theorem with initial/terminal condition $\bar\psi$ and $\bar\sigma$
respectively. We have the following properties of the propagators.

\begin{proposition}\label{prop:propagator}
For any $0\leq s\leq r \leq t\leq T$ we have
\begin{align*}
  P_{t,s}=P_{t,r}\circ P_{r,s}\;,\quad \hat P_{t,s} = \hat P_{r,s}\circ \hat P_{t,r}\;.
\end{align*}
Moreover, for $\bar\psi\in \R^{\cX_s}$, $\bar\sigma\in\R^{\cX_t}$ we have
\begin{align}\label{eq:adjoint}
  \ip{P_{t,s}\bar\psi,\bar\sigma} = \ip{\bar\psi,\hat P_{t,s}\bar\sigma}\;.
\end{align}
We have the maximum principle, i.e~we have 
\begin{align*}
  \min_{y\in\cX_s}\bar\psi(y) \leq P_{t,s}\bar\psi(x) \leq \max_{y\in\cX_s} \bar\psi(x) \quad\forall x\in\cX_t\;,\\
  \min_{y\in\cX_t}\bar\sigma(y) \leq \hat P_{t,s}\bar\sigma(x) \leq \max_{y\in\cX_t} \bar\sigma(x) \quad\forall x\in\cX_s\;.
\end{align*}
Moreover, $P_{t,s}\bar\psi$ and $\hat P_{t,s}\bar\sigma$ are strictly
positive provided $\bar\psi$ and $\bar\sigma$ are non-negative and not
identically $0$. Finally, we have that
\begin{align*}
  \sum_{x\in\cX_s} \hat P_{t,s}\bar\sigma(x) = \sum_{x\in\cX_t} \bar\sigma(x)\;.
\end{align*}
In particular, for $\bar\mu\in\cP(\cX_t)$ we have
$\hat P_{t,s}\mu\in\cP(\cX_s)$.
\end{proposition}

\begin{proof}
  These properties follow immediately from the corresponding
  properties during each interval $I_i$ established
  Section\ref{sec:heat-simple}, in particular Lemma
  \ref{lem:max-principle}, together with the boundary conditions
  \eqref{eq:heat-eq-cont}, \eqref{eq:heat-eq-adj-cont}.
\end{proof}

The asymptotics of solutions at singular times can be described
in more detail.
\begin{proposition}\label{prop:asymptotics}
  We have that
  \begin{align}\label{eq:asymptotic1}
    P_{t_{i+1},s}\bar\psi (z) &= \sum_{x\in c_i^{-1}(z)} \bar\psi(x)\bar\pi_i^{c,z}(x) + O(|t_{i+1}-s|)\;,\\\label{eq:asymptotic2}
  P_{t,t_i}\bar\psi(x) &= \bar\psi (z) + O(|t-t_i|)\;,\quad x\in s_i^{-1}(z)\;.
  \end{align}
Similarly, for the adjoint equation, we have
\begin{align}\label{eq:asymptotic3}
    \hat P_{t_{i+1},s}\bar\sigma (x) &= \bar\sigma(z)\bar\pi_i^{c,z}(x) + O(|t_{i+1}-s|)\;,\quad x\in c_i^{-1}(z)\;,\\\label{eq:asymptotic4}
\hat P_{t,t_i}\bar\sigma(z) &= \sum_{x\in s_i^{-1}(z)}\bar\sigma(x) + O(|t-t_i|)\;.
\end{align}
Moreover, we have for all $x,y\in c_i^{-1}(z)$ and $t\in(t_i,t_{i+1})$
\begin{align}\label{eq:asymptotic6}
  \abs{P_{t,s}\psi(x)-P_{t,s}\psi(y)}\leq C \exp\Big(-\int_{s\wedge{t_i}}^t Q^z_*(r)\dd r\Big)\;,
\end{align}
for a suitable constant $C$ depending on $\psi$ and $\pi^z_*$, where
$Q_*^z(r)=\min\{ Q_r(x,y):x,y\in c_i^{-1}(z),Q^c_i(x,y)=\infty\}$ and
$\pi_*^z=\inf\{\pi_r(x):x\in c_i^{-1}(z), r\in (s\wedge
t_{i},t_{i+1})\}>0$.
An analogous estimate holds for the density
$(\hat P_{t,s}\sigma)/\pi_s$ as $s\downarrow t_i$ .

Finally, we have
\begin{align}\label{eq:asymptotic5}
\lim_{s\downarrow t_i}\hat P_{t,s}\bar\sigma (x) &= \hat P_{t,t_i}\bar\sigma(z)\bar\pi_i^{s,z}(x)\;,\quad x\in s_i^{-1}(z)\;.
  \end{align}
\end{proposition}

The proof of Proposition \ref{prop:asymptotics} will follow alongside
the one of Theorem \ref{thm:heat-eq}.

\begin{proof}[Proof of Theorem \ref{thm:heat-eq}]
  It suffices to consider the case of a single interval
  $0=t_0<t_1=T$. The general case with multiple intervals and singular
  times then follows immediately by concatenating solutions on
  different intervals. For simplicity, we write $s=s_0, c=c_0$.

  {\bf Step 1:} Recall that for $t_0<s<t_1$ and
  $\tilde \psi\in \R^{\cX_0}$ there exists a unique solution
  $\psi=P_{\cdot,s}\tilde\psi$ to the heat equation on
  $[s,t_1)\times \cX_0$ with $\psi(s,\cdot)=\tilde\psi$. We will show
  that for all $z\in\bar\cX_1$ and $x\in c^{-1}(z)$ the limit
  $\psi^{c,z}$ of $\psi(t,x)$ as $t\uparrow t_1$ exists and is
  independent of $x$.  This will allow to define the propagator
  $P_{s,t_1}:\R^{\cX_0}\to\R^{\bar\cX_1}$ by setting
  $P_{s,t_1}\tilde\psi(z)=\psi^{c,z}$. Obviously, this way,
  $P_{s,t_1}$ will still be linear and satisfy the propagator identity
  and the maximum principle.

By the maximum principle, Lemma \ref{lem:max-principle}, $\psi$ is
uniformly bounded on $[s,t_1)\times\cX_0$. Assume first that
$c^{-1}(z)=\{x\}$ is a singleton. Then $Q_t(x,y)$ is uniformly bounded
on $[s,t_1)$ and has a limit as $t\uparrow t_1$ for all
$y\in\cX_0$. From the heat equation
\begin{align*}
 \partial_t\psi(t,x)=\sum_y\big[\psi(t,y)-\psi(t,x)\big]Q_t(x,y)
\end{align*}
we thus infer that $t\mapsto\psi(t,x)$ is Lipschitz on $[s,t_1)$ and
thus has a limit as $t\uparrow t_1$.

Assume now that $c^{-1}(z)$ is not a singleton and
put 
\begin{align*}
m(t):=\sum_{x\in c^{-1}(z)}\psi(t,x)\pi_t(x)\;,\quad 
v(t):= \sum_{x\in c^{-1}(z)}\big|\psi(t,x)-m(t)\big|^2\pi_t(x)\;.  
\end{align*}
We calculate
\begin{align*}
  \frac{\dd}{\dd t}m(t)
   &=
  \sum_{x\in c^{-1}(z)}\Delta_t\psi(t,x)\pi_t(x) + \psi(t,x)\dot\pi_t(x)\\
&=
  \sum_{x,y\in c^{-1}(z)}\big[\psi(t,y)-\psi(t,x)\big]Q_t(x,y)\pi_t(x)\\
  &\quad+ \sum_{x\in c^{-1}(z), y\notin c^{-1}(z)}\big[\psi(t,y)-\psi(t,x)\big]Q_t(x,y)\pi_t(x)\\
&\quad + \sum_{x\in c^{-1}(z)}\psi(t,x)\dot\pi_t(x)\;.
\end{align*}
The first sum vanishes by the detailed balance condition. In the
second sum $Q_t(x,y)$ remains bounded as $t\uparrow t_1$. Together
with the maximum principle and the assumption that $\pi_t$ is
Lipschitz we infer that $t\mapsto m(t)$ is Lipschitz and the limit 
$m(t_{1}):=\lim_{t\uparrow t_1}m(t)$ exists.

Similarly, using the detailed balance condition we calculate
\begin{align*}
  \frac{\dd}{\dd t} v(t)
 &=
\sum_{x\in c^{-1}(z)}2\big[\psi(t,x)-m(t)\big]\big[\Delta_t\psi_t(x)-\dot m(t)\big]\pi_t(x) +\big|\psi(t,x)-m(t)\big|^2\dot\pi_t(x) \\
&=  
\sum_{x,y\in c^{-1}(z)}-\big[\psi(t,y)-\psi(t,x)\big]^2Q_t(x,y)\pi_t(x)\\
  &\quad+ 2\sum_{x\in c^{-1}(z), y\notin c^{-1}(z)}\big[\psi(t,x)-m(t)\big]\big[\psi(t,y)-\psi(t,x)\big]Q_t(x,y)\pi_t(x)\\
&\quad -2 \sum_{x\in c^{-1}(z)}\big[\psi(t,x)-m(t)\big]\dot m(t)\pi_t(x)+\big|\psi(t,x)-m(t)\big|^2\dot\pi_t(x)\;.
\end{align*}
As before the terms in the last two lines are uniformly bounded by
some constant $C$ as $t\uparrow t_1$. On the other hand, one readily
checks by expanding the square that
\begin{align*}
  \sum_{x,y\in c^{-1}(z)}\big[\psi(t,y)-\psi(t,x)\big]^2Q_t(x,y)\pi_t(x) \geq 2 Q_*(t) v(t)\;,
\end{align*}
where $Q_*(t)$ is maximal such that $Q_t(x,y)\geq Q_*(t)$ for all
$x,y\in c^{-1}(z)$ with $Q^c_0(x,y)=\infty$. Note that since
$\pi^z_*=\inf\{\pi_r(x):x\in c^{-1}(z), r\in(t_0,t_1)\}>0$ by
assumption, we have that $\int^{t_{1}}Q_*(t)\dd t=+\infty$. Thus, we
have $\dot v(t) \leq -2Q_*(t)v(t) + C$ and Gronwall's lemma implies
that
\begin{align*}
  v(t) \leq \big(v(s)+ C(t-s)\big)\exp\Big(-2\int_s^{t}Q_*(r)\dd r\Big)\to 0\quad \text{as }t\uparrow t_{1}\;.
\end{align*}
We conclude that $\psi(t,x)$ converges to $m(t_1)$ for all
$x\in c^{-1}(z)$ as $t\uparrow t_1$. In particular, using that
$\abs{\psi(t,y)-\psi(t,x)}^2\leq 4v(t)/\pi_*^z$, we have established
\eqref{eq:asymptotic6}.

Finally, let us show in addition that for $\tilde\psi=\delta_x$ for $x\in c^{-1}(z)$ for some $z\in \bar\cX_1$ we have
\begin{align}\label{eq:psi-coll-delta}
  P_{t_1,t}\tilde\psi = \bar\pi^{c,z}(x)\delta_z + O(|t_1-t|)\;.
\end{align}
and thus also \eqref{eq:asymptotic1} by
linearity.

Indeed, for this $\tilde\psi$ we have $m(t)=\pi_t(x)$. Since $m$ and $\pi$ are Lipschitz we have
\begin{align*}
  P_{t_1,t}\tilde\psi(z)&=\lim_{r\uparrow t_1}\sum_{x\in c^{-1}(z)}\bar\pi^{c,z}(x)\psi(r,x)
= \lim_{r\uparrow t_1}\pi_r\big(c^{-1}(z)\big)^{-1}m(r)\\
&= \pi_t\big(c^{-1}(z)\big)^{-1}m(t) +O(|t_1-t|)
= \bar\pi^{c,z}(x)  +O(|t_1-t|)\;.
\end{align*}
Arguing similarly, we show that for $z'\neq z$ we have $P_{t_1,t}\tilde\psi(z')=0 +O(|t_1-t|)$.

\smallskip

{\bf Step 2:} Now, we fix $t_0<t<t_1$ and
$\tilde\sigma\in \cP(\cX_0)$. Recall that there exist a unique
solution $\sigma=\hat P_{t,\cdot}\tilde \sigma$ to the adjoint heat
equation on $(t_0,t]\times\cX_0$ with
$\sigma(t,\cdot)=\tilde\sigma$. We will show that for all
$z\in\bar\cX_0$ and $x\in s^{-1}(z)$ the limit
$\sigma^s(x):=\lim_{s\downarrow t_0}\sigma(s,x)$ exists in $(0,1)$ and
that we have
\begin{align}\label{eq:mu-spawn-eq}
  \sigma^s(x) = \bar\pi^{s,z}(x)\sigma^{s,z}\;, \quad \sigma^{s,z}:=\sum_{x\in s^{-1}(z)}\sigma^s(x)\;.
\end{align}
This will allow to define the propagator
$\hat P_{t,t_0}:\R^{\cX_0}\to\R^{\bar\cX_0}$ by setting
$\hat P_{t,t_0}\tilde\sigma(z)=\sigma^{s,z}$. Obviously, this way,
$\hat P_{t,t_0}$ will still be linear and satisfy the propagator
identity and the maximum principle. Moreover, we obtain
\eqref{eq:asymptotic5}.
 
If $s^{-1}(z)=\{x\}$ is a singleton, we infer similarly as in the
first step, that $s\mapsto\sigma(s,x)$ is Lipschitz on $(t_0,t]$ and
thus the limit $\sigma^s(x)$ exists.

Assume that $s^{-1}(z)$ is not a singleton. Note that the density
$\rho(s,x):=\sigma(s,x)/\pi_s(x)$ satisfies the adjoint heat equation
\begin{align*}
  \partial_s\rho(s,x) = -\Delta_s\rho(s,x) - \dot p_s(x)\rho(s,x)\;.
\end{align*}
By the assumptions the second term remains bounded as
$s\downarrow t_0$. Reversing time we can thus argue as in the first
step to see that $\rho(s,x)$ converges to a constant $\bar\rho$ as
$s\downarrow t_0$ independent of $x$. Since $\pi_s$ has a limit
$\pi^s(x)$ we infer that
$\lim_{s\downarrow t_0}\sigma(s,x)=\bar\rho\pi^s(x)$, which immediately implies \eqref{eq:mu-spawn-eq}.

Finally, let us show in addition that for $\tilde\sigma=\delta_x$ for some $x\in s^{-1}(z)$, $z\in\bar\cX_0$ we have 
\begin{align}\label{eq:mu-spawn-delta}
  \hat P_{t,t_0}\delta_x = \delta_z + O(|t-t_0|)\;,
\end{align}
and thus \eqref{eq:asymptotic4} by linearity.

Let us put
\begin{align*}
  m(s):= \sum_{x\in s^{-1}(z)}\sigma(s,x)\;.
\end{align*}
We have $m(t)=1$ and $\lim_{s\downarrow t_0}m(s)= \hat P_{t,t_0}\tilde\sigma(z)$. We calculate
\begin{align*}
  \frac{\dd}{\dd s} m(s) 
  &=
   \sum_{x\in s^{-1}(z)}-\hat\Delta_s\sigma(s,x)
  = 
   \sum_{x\in s^{-1}(z),~y\in\cX_0,~y\neq x}-\sigma(s,y)Q_s(y,x) + \sigma(s,x)Q_s(x,y)\\
  &=  
\sum_{x,y\in s^{-1}(z),~y\neq x}-\sigma(s,y)Q_s(y,x) + \sigma(s,x)Q_s(x,y)\\
 &\quad + \sum_{x\in s^{-1}(z),~y\notin s^{-1}(z)}-\sigma(s,y)Q_s(y,x) + \sigma(s,x)Q_s(x,y)\;.
\end{align*}
The first sum in the right hand side vanishes by symmetry. In the
second sum $Q_s(x,y)$ remains bounded as $s\downarrow t_0$. Thus
$s\mapsto m(s)$ is Lipschitz which yields \eqref{eq:mu-spawn-delta}.

\smallskip

{\bf Step 3:} We show that given $\bar\psi\in\bar\cX_0$ there exist a
unique solution $\psi$ on $(t_0,t_1)\times\cX_0$ such that
\begin{align*}
  \bar\psi(z)&=\lim_{t\downarrow t_0}\psi(t,x)\quad\forall z\in\bar\cX_0,~x\in s^{-1}(z)\;.
\end{align*}
This will allow to define the propagator $P_{t,t_0}$ for all $t\in [t_0,t_1]$.

To show uniqueness let $\psi$ be any such solution. Then for any
$t_0<s<t<t_1$ and $\tilde\sigma\in\R^{\cX_0}$ we have
\begin{align*}
  \ip{\psi(t,\cdot),\tilde\sigma} = \ip{\psi(s,\cdot),\hat P_{t,s}\tilde\sigma} \overset{s\downarrow t_0}\longrightarrow \sum_{z\in\bar\cX_0} \bar\psi(z)\hat P_{t,t_0}\tilde\sigma\;,
\end{align*}
using the assumption on $\psi$ and the convergence of the solution to
the adjoint equation from step 2. Thus the solution $\psi$ is uniquely
determined. To show existence, we define $\psi(t,\cdot)$ via
$\ip{\psi(t,\cdot),\tilde\sigma}=\ip{\bar\psi,\hat
  P_{t,t_0}\tilde\sigma}$
for $\tilde \sigma\in \R^{\cX_0}$. Using \eqref{eq:mu-spawn-delta} we
see that $\psi$ has the correct limit as $t\downarrow t_0$. It remains
to verify that it is a solution. To this end it suffices to show that
extending \eqref{eq:other-index} for $t_0<t<t_1$ we have 
 \begin{align}\label{eq:other-index2}
   \partial_t\hat P_{t,t_0}\tilde\sigma=\hat P_{t,t_0}\hat\Delta_t\tilde\sigma\;. 
 \end{align}
Indeed, from this we obtain immediately
\begin{align*}
  \frac{\dd}{\dd t}\ip{\psi(t,\cdot),\tilde\sigma}
= \ip{\bar\psi,\hat P_{t,t_0}\hat\Delta_t\tilde\sigma}
= \ip{\psi(t,\cdot),\hat\Delta_t\tilde\sigma}
 = \ip{\Delta_t\psi(t,\cdot),\tilde\sigma}\;.
\end{align*}
Let us show \eqref{eq:other-index2}. For $t_0<s<t$ we obtain integrating \eqref{eq:other-index}
\begin{align*}
  \hat P_{t+h,s}\tilde\sigma -\hat P_{t,s}\tilde\sigma = \int_t^{t+h}\hat P_{r,s}\hat\Delta_r\tilde \sigma \dd r\;.
\end{align*}
Noting that the rates $Q_r$ are bounded for $r\in[t,t+h]$ and thanks
to the maximum principle we can thus first pass to the limit
$s\downarrow t_0$ by dominated convergence. Again thanks to the
maximum principle, linearity, and the continuity assumption on the
rates, the map $r\mapsto \hat P_{r,t_0}\hat\Delta_r\tilde\sigma(x)$ is
continuous. Thus we can divide by $h$ and let $h\downarrow 0$ to
obtain the claim (arguing similarly for the left derivative).

\smallskip

{\bf Step 4:} 
Similarly, we show that given $\bar\sigma\in\bar\cX_1$ there exist a unique solution $\sigma$ on $(t_0,t_1)\times\cX_0$ such that 
\begin{align*}
  \bar\sigma(z)&=\sum_{x\in c^{-1}(z)}\lim_{s\uparrow t_1}\sigma(s,x)\quad\forall z\in\bar\cX_1\;.
\end{align*}
This will allow to define the propagator $\hat P_{t_1,s}$ for all $s\in [t_0,t_1]$.

To show uniqueness let $\sigma$ be any such solution. Then for any
$t_0<s<t<t_1$ and $\tilde\psi\in\R^{\cX_0}$ we have
\begin{align*}
  \ip{\tilde\psi,\sigma(s,\cdot)} = \ip{P_{t,s}\tilde\psi,\sigma(t,\cdot)} \overset{t\uparrow t_1}\longrightarrow \sum_{z\in\bar\cX_1} P_{t_1,s}\tilde\psi(z)\bar\sigma(z)\;,
\end{align*}
using the assumption on $\sigma$ and the convergence of the solution
to the heat equation from step 1. Thus the solution $\mu$ is uniquely
determined. To show existence we define $\sigma(s,\cdot)$ via
$\ip{\tilde\psi(s,\cdot),\sigma(s,\cdot)}=\ip{P_{t_1,s}\tilde\psi,\bar\mu}$
for $\tilde \psi\in \R^{\cX_0}$. Using \eqref{eq:psi-coll-delta} shows
that this solution has the correct limit as $s\uparrow t_1$.
Similarly as before one can show that this is a solution to the heat
equation by showing that
$\partial_s P_{t_1,s}\tilde\psi=-P_{t_1,s}\Delta_s\tilde\psi$
extending \eqref{eq:other-index}.

Note that the propagators $P_{t,t_0}$ and $\hat P_{t_1,s}$ constructed
in steps 3 and 4 by construction satisfy the adjointness relation
\eqref{eq:adjoint}.

Finally, note that \eqref{eq:asymptotic2} and \eqref{eq:asymptotic3}
follow from \eqref{eq:asymptotic4} and \eqref{eq:asymptotic1} by the
adjointness \eqref{eq:adjoint}.
\end{proof}

\section{Characterizations of super Ricci flows}
\label{sec:RFequiv}
In this section we will give several equivalent characterizations of
discrete super Ricci flows. These will be formulated in terms of a
time-dependent Bochner inequality, gradient estimates for the heat
propagator, transport estimates for the dual heat propagator, and
dynamic convexity of the entropy.

Throughout this section $(\cX_t,Q_t,\pi_t)_{t\in[0,T]}$ will be a
singular time-dependent Markov triple according to Definition
\ref{def:sing-triple}. We additionally make the following assumption
on the growth of the transition rates that go to infinity at singular
times: For each $z\in \bar\cX_{i+1}$ we assume that
 \begin{align}\label{eq:rate-int-infty2}
 Q_{i,\max}^{z,c}(t)\exp\Big(-2\int^tQ^{z,c}_{i,\min}(r)\, dr\Big)\to0 \quad \text{as } t\nearrow t_{i+1}\; ,
 \end{align}
 where we set
 $Q_{i,\max}^{z,c}(t)=\max\{Q_t(x,y) : x,y\in c_i^{-1}(z),
 Q^c_i(x,y)=\infty\}$
 and
 $Q^{z,c}_{i,\min}(t)=\min\{Q_t(x,y): x,y\in c_i^{-1}(z),
 Q^c_i(x,y)=\infty\}$
 are the maximal resp.~minimal diverging rates in a collapsing
 region. Note that $Q^{z,c}_{i,\min}(t)\to\infty$ as
 $t\nearrow t_{i+1}$.

 Moreover for all $z\in\bar \cX_{i}$ we assume that
  \begin{align}\label{eq:rate-int-infty3}
 (t-t_i)^2Q^{z,s}_{i,\max}(t)\to 0\quad \text{as } t\searrow t_i\; ,
 \end{align}
 where $Q^{z,s}_{i,\max}(t)=\max\{Q_t(x,y) : x,y\in s_i^{-1}(z), Q^s_i(x,y)=\infty\}$.

\medskip

 To state the defining properties of super Ricci flows, let us
 introduce or recall the following central objects. We will denote by
 $\Gamma_t$ and $\Gamma_{2,t}$ the integrated carr\'e du champs
 operator associated to the Markov triple $(\cX_t,Q_t,\pi)$,
 c.f.~\eqref{eq:Gamma}, \eqref{eq:Gamma2}, i.e.
 \begin{align*}
  \Gamma_t(\mu,\psi)&:=\ip{\nabla\psi,\nabla\psi\cdot\Lambda_t(\mu)}\;,\\
  \Gamma_{2,t}(\mu,\psi)&:=\frac12\ip{\nabla\psi,\nabla\psi\cdot\hat\Delta_t \Lambda_t(\mu)} -\ip{\nabla\psi,\nabla\Delta_t\psi\cdot\Lambda_t(\mu)}\;,
 \end{align*}
 where we write
 $\Lambda_t(\mu)(x,y)=\Lambda\big(\mu(x)Q_t(x,y),\mu(y)Q_t(y,x)\big)$
 and $\hat\Delta_t\Lambda_t(\mu)$ is defined as in Section
 \ref{sec:BA}. Moreover, we introduce the time-derivative of the
 $\Gamma$-operator given by
\begin{align}\label{eq:gamma-del}
  \partial_t\Gamma_t(\mu,\psi) := \ip{\nabla\psi,\nabla\psi\cdot\partial_t\Lambda_t(\mu)}\;,
\end{align}
where we set
\begin{align*}
\partial_t\Lambda_t(\mu)(x,y)=& \partial_1\Lambda\big(\mu(x)Q_t(x,y),\mu(y)Q_t(y,x)\big)\mu(x)\dot Q_t(x,y)\\
 &+ \partial_2\Lambda\big(\mu(x)Q_t(x,y),\mu(y)Q_t(y,x)\big)\mu(y)\dot Q_t(y,x)\;.
\end{align*}
Note that by the Lipschitz assumption on the transition rates,
$\partial_t\Gamma_t$ is well defined for a.e.~$t\in(0,T)$. Further let
us denote by $\cW_t$ the discrete transport distance associated to
$(\cX_t,Q_t,\pi_t)$. Finally, we denote by $\cH_t$ the relative entropy w.r.t.~$\pi_t$.

With this we have the following result.

\begin{theorem}\label{thm:srf-equiv}
  Let $(\cX_t,Q_t,\pi_t)_{t\in[0,T]}$ be a singular time-dependent
  Markov triple satisfying \eqref{eq:rate-int-infty2} and
  \eqref{eq:rate-int-infty3}. Then the following are equivalent
\begin{enumerate}
\item [(I)] The \emph{Bochner inequality}
\begin{align}\label{eq:Bochner}
\Gamma_{2,t}(\mu,\psi)\geq\frac12\partial_t\Gamma_t(\mu,\psi)
\end{align}
holds for a.e.~$t\in[0,T]$ and all $\mu\in\cP(\cX_t)$, $\psi\in\R^{\cX_t}$.
\item[(II)] The \emph{gradient estimate}
\begin{align}\label{eq:grad-est}
\Gamma_t(\mu,P_{t,s}\psi)\leq\Gamma_s(\hat P_{t,s}\mu,\psi)
 \end{align}
  holds for all $0\leq s\leq t\leq T$ and all $\mu\in\cP(\cX_t)$, $\psi\in\R^{\cX_s}$.
\item[(III)] The \emph{transport estimate}
\begin{align}\label{eq:transport}
\cW_s(\hat P_{t,s}\mu,\hat P_{t,s}\nu)\leq \cW_t(\mu,\nu)
 \end{align}
  holds for all $0\leq s \leq t\leq T$ and all $\mu,\nu\in\cP(\cX_t)$.
\item[(IV)] The \emph{entropy is dynamically convex}, i.e.~for
  a.e.~$t\in[0,T]$ and all $\cW_t$-geodesics $(\mu^a)_{a\in[0,1]}$
\begin{align}\label{eq:dyn-convex}
\partial_a^+\cH_t(\mu^{1-})-\partial_a^-\cH_t(\mu^{0+})\geq-\frac12\partial_t^-\cW_{t-}^2(\mu^0,\mu^1)\;.
\end{align}
\end{enumerate}
\end{theorem}

\begin{definition}\label{def:srf}
  A time-dependent Markov triple $(\cX_t,Q_t,\pi_t)_t$ is called a
  \emph{super Ricci flow} if any of the equivalent properties of the
  previous theorem holds.
\end{definition}

The proof of Theorem \ref{thm:srf-equiv} will be given in the
following subsections. We will show the following implications:
$(I)\Leftrightarrow(II)$, $(II)\Leftrightarrow(III)$,
$(IV)\Rightarrow(I)$, $(II)$ implies the dynamic EVI$^-$ property of
the heat flow, which together with $(III)$ implies $(IV)$.

\subsection{Bochner formula and gradient estimates}
In this section we prove the implication $(I)\Leftrightarrow(II)$.

\medskip

\begin{proof}[Proof of $(I)\Rightarrow(II)$] ~\\
{\bf Step 1:} We will first show that \eqref{eq:grad-est} holds for $t_i<s\leq t<t_{i+1}$.
  To this end, fix $\mu\in\cP(\cX_t)$ and $\psi\in\R^{\cX_s}$ and for $s\leq r\leq t$ set $\mu_r=\hat P_{t,r}\mu$ and
  $\psi_r=P_{r,s}\psi$. Then we have
  \begin{align*}
  &\frac{d}{dr}\Gamma_r(\mu_r,\psi_r)\\
  &=\sum_{x,y}\nabla\psi_r \nabla\Delta_r\psi_r\Lambda_r(\mu_r)(x,y)\\
  &\quad+\frac12\sum_{x,y}|\nabla\psi_r|^2\Big[-\partial_1\Lambda_r(\mu_r)\hat\Delta_r\mu_r(x)Q_r(x,y)-\partial_2\Lambda_r(\mu_r) \hat\Delta_r\mu_r(y) Q_r(y,x)\Big]\\
  &\quad+\frac12\sum_{x,y}|\nabla\psi_r|^2\Big[\partial_1\Lambda_r(\mu_r)\mu_r(x)\dot Q_r(x,y)+\partial_2\Lambda_r(\mu_r) \mu_r(y) \dot Q_r(y,x)\Big]\;,
 \end{align*}
 where we have put for brevity
 $\partial_1\Lambda_r(\mu)(x,y)=\partial_1\Lambda\big(\mu(x)Q_r(x,y),\mu(y)Q_r(y,x)\big)$
 and similarly for $\partial_2\Lambda_r(\mu)$. Inserting the
 definition of $\Gamma_{2,r}$ and $\partial_r\Gamma_r$ we obtain
\begin{align*}
  \frac{d}{dr}\Gamma_r(\mu_r,\psi_r)
  =-2\Gamma_{2,r}(\mu_r,\psi_r)+\partial_r\Gamma_r(\mu_r,\psi_r)\leq 0\;,
 \end{align*}
 where the last inequality follows from
 \eqref{eq:Bochner}. Integrating over $r\in(s,t)$ then yields the
 gradient estimate \eqref{eq:grad-est}.

 \medskip {\bf Step 2:} Now, we establish \eqref{eq:grad-est} across a
 singular time, i.e.~for $t_{i-1}<s<t_i<t<t_{i+1}$. This then readily
 implies \eqref{eq:grad-est} for all $s,t$. From the previous step we
 obtain for $\eps>0$ sufficiently small
\begin{align*}
\Gamma_t(\mu,P_{t,s}\psi)\leq& \Gamma_{t_i+\eps}(\hat P_{t,t_i+\eps}\mu, P_{t_i+\eps,s}\psi)\;,\\
\Gamma_s(\hat P_{t,s}\mu,\psi)\geq& \Gamma_{t_i-\eps}(\hat P_{t,t_i-\eps}\mu,P_{t_i-\eps,s}\psi)\;.
\end{align*}
Thus, it will be sufficient to show
\begin{align}\label{eq:Gamma-cont}
\lim_{\eps\to0} \Gamma_{t_i-\eps}(\hat P_{t,t_i-\eps}\mu,P_{t_i-\eps,s}\psi)=\Gamma_{t_i}(\hat P_{t,t_i}\mu,P_{t_i,s}\psi)
=\lim_{\eps\to0}\Gamma_{t_i+\eps}(\hat P_{t,t_i+\eps}\mu, P_{t_i+\eps,s}\psi)
\;.
\end{align}
Let us first show that the first identity in
\eqref{eq:Gamma-cont}. For this let $z\in\bar\cX_i$ and write
$\mu_\eps=\hat P_{t,t_i-\eps}\mu$, $\psi_\eps=P_{t_i-\eps,s}\psi$, $\Lambda_\eps=\Lambda_{t_i-\eps}$, and
$c=c_{i-1}$. Then, using \eqref{eq:asymptotic6}, we estimate for ${x,y\in c^{-1}(z)}$:
\begin{align*}
|\nabla\psi_\eps|^2(x,y)\Lambda_{\eps}(\mu_\eps)(x,y)
\leq C\exp\Big(-\int_{s}^{t_i-\eps}Q^{z,c}_{i,\min}(r)\, \dd r\Big)Q^{z,c}_{i,\max}(t_i-\eps)\;.
\end{align*}
Hence we find with the assumption \eqref{eq:rate-int-infty2} that for
all $z\in\bar\cX_i$
\begin{align*}
\lim_{\eps\to0}\sum_{x,y\in c^{-1}(z)}|\nabla\psi_\eps|^2(x,y)\Lambda_\eps(\mu_\eps)(x,y)=0.
\end{align*}
Moreover, for $z\neq z'\in\bar\cX_i$ with Theorem \ref{thm:heat-eq} and \eqref{eq:limit-Q} we find
\begin{align*}
&\lim_{\eps\to0}\sum_{x\in c^{-1}(z),y\in c^{-1}(z')}|\nabla\psi_\eps|^2(x,y)\Lambda_\eps(\mu_\eps)(x,y)\\
=&\sum_{x\in c^{-1}(z),y\in c^{-1}(z')}|\nabla\psi_0|^2(z,z')\Lambda\Big(\mu_0(z)\frac{\pi_i^c(x)}{\pi_{t_i}(z)}Q^c_i(x,y),\mu_0(z')\frac{\pi_i^c(y)}{\pi_{t_i}(z')}Q^c_i(y,x)\Big)\;.
\end{align*}
From the positive $1$-homogeneity of $\Lambda$ we have that
$\Lambda(r, s)+\Lambda(r',s')=\Lambda(r+r',s+s')$ whenever
$r=\lambda r'$ and $s=\lambda s'$ for some $\lambda\geq0$. Since we
have $Q^c_i(x,y)\pi^c_i(x)=Q^c_i(y,x)\pi^c_i(y)$, we deduce that
\begin{align*}
&\lim_{\eps\to0}\sum_{x\in c^{-1}(z),y\in c^{-1}(z')}|\nabla\psi_\eps|^2(x,y)\Lambda_\eps(\mu_\eps)(x,y)\\
=&|\nabla\psi_0|^2(z,z')\Lambda\Big(\mu_0(z)\sum_{x\in c^{-1}(z),y\in c^{-1}(z')}\frac{\pi_i^c(x)}{\pi_{t_i}(z)}Q^c(x,y),\mu_0(z')\sum_{x\in c^{-1}(z),y\in c^{-1}(z')}\frac{\pi_i^c(y)}{\pi_{t_i}(z')}Q^c(y,x)\Big)\\
=&|\nabla\psi_0|^2(z,z')\Lambda\Big(\mu_0(z)Q_{t_i}(z,z'),\mu_0(z')Q_{t_i}(z',z)\Big),
\end{align*}
where we used again \eqref{eq:limit-Q}. Summing over all
$z\neq z'\in \bar\cX_i$ yields the first identity in
\eqref{eq:Gamma-cont}.

Let us now show the second identity. We write $s_i=s$ and
$\mu_\eps=\hat P_{t,t_i+\eps}\mu$,
$\psi_\eps=P_{t_i+\eps,s}\psi$, and $\Lambda_\eps=\Lambda_{t_i+\eps}$. Then, by
\eqref{eq:asymptotic2} we obtain for $z\in\bar \cX_i$ and $x,y\in s^{-1}(z)$:
\begin{align*}
&|\nabla\psi_\eps|^2(x,y)\Lambda_{\eps}(\mu_\eps)(x,y)\leq 4C\eps^2 Q^{z,s}_{i,\max}({t_i+\eps})
\end{align*}
for some constant $C$. Hence we deduce from \eqref{eq:rate-int-infty3} that 
\begin{align*}
\lim_{\eps\to0}\sum_{x,y\in s^{-1}(z)}|\nabla\psi_\eps|^2(x,y)\Lambda_{t_i+\eps}(\mu_\eps)(x,y)=0\;.
\end{align*}
For $z\neq z'\in\bar\cX_i$ with Theorem \ref{thm:heat-eq} and  \eqref{eq:asymptotic5} we find similarly as above, that
\begin{align*}
&\lim_{\eps\to0}\sum_{x\in s^{-1}(z),y\in s^{-1}(z')}|\nabla\psi_\eps|^2(x,y)\Lambda_{t_i+\eps}(\mu_\eps)(x,y)\\
&=|\nabla\psi_0|^2(z,z')\sum_{x\in s^{-1}(z),y\in s^{-1}(z')}\Lambda\Big(\mu_0(z)\frac{\pi^{s,z}}{\pi_{t_i}(y)}Q^s_i(x,y),\mu_0(z')\frac{\pi^{s}_i(y)}{\pi_{t_i}(z')}Q^s(y,x)\Big)\\
&=|\nabla\psi_0|^2(z,z')\Lambda\Big(\mu_0(z)Q_{t_i}(z,z'),\mu_0(z')Q_{t_i}(z',z)\Big)\;.
\end{align*}
Summing over all $z,z'\in\bar \cX_{i}$ yields the second identity in \eqref{eq:Gamma-cont}. This finishes the proof.
\end{proof}

\medskip

\begin{proof}[Proof of $(II)\Rightarrow(I)$]~\\
  Consider $t_i< s<t<t_{i+1}$ for some $i$ and set again  $\mu_r=\hat P_{t,r}\mu$ and
  $\psi_r=P_{r,s}\psi$ for $\mu\in\cP(\cX_i)$ and $\psi\in\R^{\cX_i}$. Arguing similarly as before, we find
\begin{align*}
  0\geq\Gamma_t(\mu,P_{t,s}\psi)- \Gamma_s(\hat P_{t,s}\mu,\psi)=\int_s^t\frac{d}{dr}\Gamma_r(\mu_r,\psi_r)\ddd r
  =\int_s^t-2\Gamma_{2,r}(\mu_r,\psi_r)+\partial_r\Gamma_r(\mu_r,\psi_r)\ddd r.
\end{align*}
Dividing by $t-s$ and letting $s\to t$ the Lebesgue differentiation
theorem implies that for a.e.~$t\in(t_i,t_{i+1})$ we have
\begin{align*}
  \Gamma_{2,t}(\mu,\psi)\geq\frac12 \partial_t\Gamma_t(\mu,\psi)\;,
\end{align*}
which proves the claim.
\end{proof}

\subsection{Transport estimates}
\label{sec:trapo}
In this section, we will prove the implication $(II)\Leftrightarrow(III)$.

To this end, we will use the dual characterization of the discrete
transport distance given by Theorem \ref{thm:dual}. We denote by
$\HJ_{\cX_t}^1$ the set of Hamilton--Jacobi subsolutions on the
interval $[0,1]$ for the triple $(\cX_t,Q_t,\pi_t)$. Further, we need
an observation on the that metric tensor can be expressed as a limit
of distances. For this, recall from Section \ref{sec:BA} that on a
Markov triple $(\cX,Q,\pi)$ the metric $\cW$ is induced by a
Riemannian metric tensor on $\cP_*(\cX)$ which is given for
$\mu\in\cP_*(\cX)$ by
$g_\mu(s)=\ip{\nabla \psi,\nabla
  \psi}_\mu=\ip{\nabla\psi,\nabla\psi\cdot\Lambda(\mu)}$,
where we have identified the tangent space $\mathcal T$ with the space
of discrete gradients $\mathcal G$ (resp.~the set $\mathcal G'$ of
functions modulo constants) via the map
\begin{align*}
 s= K_\mu\psi = -\nabla\cdot\big(\nabla\psi\cdot\Lambda(\mu)\big)\;.
\end{align*}
In other words, we have $g_\mu(s)=\ip{s,K_\mu^{-1}s}=:G(\mu,s)$. Note
that for any $\omega\in \mathcal T$, $\psi\in\mathcal G'$, and
$\mu\in\cP_*(\cX)$ we have
\begin{align}\label{eq:Lagrange}
  \Gamma(\mu,\psi) = \ip{\nabla\psi,\nabla\psi\cdot\Lambda(\mu)} 
  = \ip{\psi, K_\mu \psi} 
   \geq 2\ip{\omega,\psi} - \ip{\omega,K_\mu^{-1}\omega}
  =  2\ip{\omega,\psi} - G(\mu,\omega)\;,
\end{align}
and we have equality if $\omega=K_\mu\psi$. The following is a direct
consequence of the observation that $\cW$ is the Riemannian distance
associated to $g_\mu$.

\begin{lemma}\label{lem:asymptotic-W}
  For any $C^1$-curve $(\mu^a)_{a\in[0,1]}$ in $\cP_*(\cX)$ we have
  \begin{align}\label{eq:asymptot}
    \lim_{a\downarrow0}\frac{1}{a^2}\cW(\mu^0,\mu^a)^2 = G(\mu^0,\dot\mu^0)\;.
  \end{align}
\end{lemma}

\begin{proof}[Proof of $(II)\Rightarrow(III)$]~\\
  Fix $0\leq s\leq t\leq T$ and let $(\phi^a)_{a\in[0,1]}$ be a Hamilton--Jacobi subsolution in
  $\HJ_{\cX_s}^1$. The gradient estimate \eqref{eq:grad-est} implies
  that $(P_{t,s}\phi^a)_a$ is again a Hamilton--Jacobi subsolution in $\HJ_{\cX_t}^1$. Indeed for any
  $\mu\in\cP(\cX_t)$ we have
  \begin{align*}
    \ip{\partial_aP_{t,s}\phi^a,\mu} 
    &=
     \ip{\dot\phi^a,\hat P_{t,s}\mu}
    \leq -\frac12\Gamma_s(\hat P_{t,s}\mu,\phi^a)
    \leq -\frac12\Gamma_t(\mu,P_{t,s}\phi^a)\;.
  \end{align*}
Thus, by the duality result Theorem \ref{thm:dual} we have
\begin{align*}
  \ip{\phi^1,\hat P_{t,s}\mu}-\ip{\phi^0,\hat P_{t,s}\nu}
 =
  \ip{P_{t,s}\phi^1,\mu}-\ip{P_{t,s}\phi^0,\nu}
\leq \frac12\cW_t(\mu,\nu)^2\;.
\end{align*}
Taking the supremum over $\phi$ and using again Theorem \ref{thm:dual} yields the claim.
\end{proof}

\medskip

\begin{proof}[Proof of $(III)\Rightarrow(II)$]~\\
  It suffices to show \eqref{eq:grad-est} for strictly positive measures, i.e.~$\mu\in\cP_*(\cX_t)$. The statement for general $\mu\in \cP(\cX_t)$ follows by approximation. Fix  $\mu\in\cP_*(\cX_t)$ and $\psi\in\R^{\cX_t}$ and put
   $\psi_r =P_{r,s}\psi$ for $r\in[s,t]$. Let $(\mu^a)_{a\in[0,1]}$ be a curve such that
\begin{align*}
\dot\mu^0+\nabla\cdot\big(\Lambda_t(\mu^0)\cdot\nabla\psi_t\big)=0\;.
\end{align*}
For instance, one could take
$\mu^a=\mu-a\eps\nabla\cdot\big(\Lambda_t(\mu)\nabla\psi_t\big)$ for
$\eps$ sufficiently small. Finally, put $\mu_r^a=\hat P_{t,r}\mu^a$
and $w_r=\dot\mu^0_r=\hat P_{t,r}\dot\mu^0$. Note that $r\mapsto \ip{w_r,\psi_r}$ is
constant. Now, we deduce from the transport estimate using \eqref{eq:asymptot} and \eqref{eq:Lagrange} that
\begin{align*}
  \Gamma_s(\mu_s,\psi_s) &\geq \ip{w_s,\psi_s}-G_s(\mu_s,w_s) =\ip{w_s,\psi_s}-\lim_{a\downarrow0}\frac1{a^2}\cW_s(\mu_s^0,\mu^a_s)^2\\
&\geq \ip{w_t,\psi_t}-\lim_{a\downarrow0}\frac1{a^2}\cW_t(\mu_t^0,\mu^a_t)^2 = \ip{w_t,\psi_t}-G_t(\mu_t,w_t)\\
&=\Gamma_t(\mu_t,\psi_t)\;.
\end{align*}
This proofs the claim.
\end{proof}

\subsection{Entropy and convexity}
In this section we prove the implication $(IV)\Rightarrow(I)$.

Let us first observe the following. Let
$(\mu^a)_{a\in(-\varepsilon,\varepsilon)}$ be a $\cW_t$-geodesic in
$\cP_*(\cX_t)$. Note that in this case we have 
\begin{align*}
  \cW_t(\mu^b,\mu^c)^2=\frac{1}{(c-b)^2}\int_b^c\norm{\nabla\psi^a}^2_{\mu^a,t}\dd a\;,
\end{align*}
$\dot\mu^a+K_{\mu^a,t}\nabla\psi^a=0$ and where
$|\nabla \psi|^2_{\mu,t}$ denotes the inner product on the tangent
space at $\mu$ associated to $(\cX_t,Q_t,\pi_t)$.  For $s\neq t$ we
have with the same choice of $\psi$
\begin{align*}
  \cW_s(\mu^b,\mu^c)^2\leq\frac{1}{(c-b)^2}\int_b^c\norm{\nabla\psi^a}^2_{\mu^a,s}\dd a\;.
\end{align*}
This implies that 
\begin{align}\label{eq:partialW}
\partial_t^-\cW^2_{t-}(\mu^b,\mu^c)\leq\frac{1}{(c-b)^2}\int_b^c-\partial_t^-\norm{\nabla\psi^a}^2_{\mu^a,t-}\dd a\;,
\end{align}
where the minus is due to the fact that $K_{\mu,t}$ is the inverse of the
metric tensor.

\medskip

\begin{proof}[Proof of $(IV)\Rightarrow(I)$]~\\
  Let $\mu\in\cP_*(\cX_t)$ and $\psi\in\R^{\cX_t}$. Let
  $(\mu^a)_{a\in(-\varepsilon,\varepsilon)}$ be the geodesic starting
  in $\mu$ with initial velocity $\nabla\psi$ and let $\psi^a$ as
  above. Then dynamic convexity together with \eqref{eq:partialW}
  implies that
\begin{align*}
\langle \Hess_t\mathcal H_t(\mu)\nabla\psi,\nabla\psi\rangle_{\mu,t}=\frac{\dd^2}{\dd a^2}\cH(\mu^a)|_{a=0}\geq\frac12\partial_t^-\norm{\nabla\psi}_{\mu,t-}^2,
\end{align*}
where we used Proposition 16.2 in \cite{Vil09}. To finish the proof,
it suffices to recall that from Section \ref{sec:BA} and the
definition of $\partial_t\Gamma_t$ that for every $t$ where
$t\mapsto Q_t$ is differentiable we have
  \begin{align*}
   \langle \Hess_t\mathcal H_t(\mu)\nabla\psi,\nabla\psi\rangle_{\mu,t}&=\Gamma_{2,t}(\mu,\psi)\;,\\
\partial_t^-|\nabla\psi|^2_{\mu,t-}&=\partial_t\Gamma_t(\mu,\psi)\;.
  \end{align*} 
\end{proof}

\subsection{Dynamic EVI}
In this section we prove the implication $(II)\Rightarrow(VI)$. More
precisely, we will show that the gradient estimate $(II)$ implies the
dynamic EVI$^-$ property for the heat flow, which together with the
transport estimate $(III)$ implies dynamic convexity $(I)$.

To this end we introduce in the spirit of \cite{sturm2016} an analogue
of the transport distance across different time slices.  We fix an
interval $I_i=(t_i,t_{i+1})$ between two singular times for some $i$ and
given $s,t\in(t_i,t_{i+1})$ for some $i$ and
$\mu^0,\mu^1\in\cP(\cX_i)$ we define
 \begin{align}\label{eq:dual-W-dyn}
  \frac12\cW_{s,t}(\mu^0,\mu^1)^2 = \sup\big\{\ip{\phi^1,\mu^1}-\ip{\phi^0,\mu^0}~:~\phi\in\HJ_{s,t}\big\}\;,
\end{align}
 where $\HJ_{s,t}$ denotes the set of all $C^1$ functions $\phi:[0,1]\to\R^{\cX_i}$ satisfying
 \begin{align}\label{eq:dual-contraint-dyn}
   \ip{\dot\phi^a,\mu} + \frac12\norm{\nabla\phi^a}_{\mu,\theta(a)}^2\leq 0
   \quad\forall \mu\in\cP(\cX_i),~a\in(0,1)\;,
 \end{align}
 where $\theta(a):=s+a(t-s)$. Note that this is not a distance in the
 usual sense since $\cW_{s,t}(\mu^0,\mu^1)\neq\cW_{s,t}(\mu^1,\mu^0)$.
 In the rest of this section we drop the index $i$ and write $\cX$
 instead of $\cX_i$ and $I$ for $I_i$.

Due to the local Lipschitz continuity in time of $Q$ and $\pi$ we have the following control.
\begin{lemma}\label{lem:log-lip}
  For each compact subinterval $J\subset I$ there exists a constant
  $L>0$ such that for all $\mu^0,\mu^1\in\PX$ and $s,t\in J$:
 \begin{align}\label{eq:log-lip-W}
    e^{-L|t-s|}\cW_{t}(\mu^0,\mu^1)^2&\leq \cW_s(\mu^0,\mu^1)^2 \leq e^{L|t-s|}\cW_t(\mu^0,\mu^1)^2\;,\\\label{eq:log-lip-W2}
   e^{-L|t-s|}\cW_{s}(\mu^0,\mu^1)^2&\leq \cW_{s,t}(\mu^0,\mu^1)^2 \leq e^{L|t-s|}\cW_s(\mu^0,\mu^1)^2\;.
 \end{align}
\end{lemma}
\begin{proof}
  Recall that by assumption the maps $r\mapsto Q_r$ and
  $r\mapsto \pi_r$ are log-Lipschitz on $J$ for some constant constant
  $L$, i.e.  for all $x,y\in\cX$ and $s,t\in I$ we have
  \begin{align}\label{eq:log-lip}
    e^{-L|t-s|}Q_s(x,y)\leq  Q_t(x,y)\leq e^{L|t-s|} Q_s(x,y)\;,\  e^{-L|t-s|}\pi_s(x)\leq  \pi_t(x)\leq e^{L|t-s|} \pi_s(x)\;.
  \end{align}
 
  The estimate \eqref{eq:log-lip-W} follows immediately from this and
  the definition of $\cW_t$.

 To show \eqref{eq:log-lip-W2} let
  $\varphi\in\HJ^1_{\cX_s}$ be a Hamilton--Jacobi subsolution with respect to
  $(\cX,Q_s,\pi_s)$. Then applying \eqref{eq:log-lip} yields
\begin{equation}\label{HopLaxstatic}
\ip{\dot\varphi^a,\mu}\leq-\frac12\norm{\nabla\varphi^a}^2_{\mu,s}\leq -\frac12 e^{-L|s-t|}\norm{\nabla\varphi^a}^2_{\mu,\theta(a)} \quad \forall\mu\in\cP(\cX),\ a\in[0,1]\; .
\end{equation}
Set $\tilde\varphi^a:=e^{-L|s-t|}\varphi^{a}$. Then $\tilde\varphi$
solves
\begin{equation*}
\ip{\dot{\tilde\varphi}^a,\mu}\leq-\frac12\norm{\nabla\tilde\varphi^a}_{\mu,\theta(a)}^2 \quad \forall\mu\in\cP(\cX),\ a\in[0,1]
\end{equation*}
and
\begin{equation*}
 e^{-L|s-t|}\left(\ip{\varphi^{1},\mu^1}-\ip{\varphi^{0},\mu^0}\right)=\ip{\tilde\varphi^{1},\mu^1}-\ip{\tilde\varphi^{0},\mu^0}\; .
\end{equation*}
Hence
\begin{equation*}
 e^{-L|s-t|}\left(\ip{\varphi^{1},\mu^1}-\ip{\varphi^{0},\mu^0}\right)\leq \frac12 \cW_{s,t}(\mu^0,\mu^1)^2.
\end{equation*}

Taking the supremum among all such $\varphi$ yields by Theorem \ref{thm:dual}
\begin{equation*}
\cW_s(\mu^0,\mu^1)^2\leq e^{L|s-t|}\cW_{s,t}(\mu^0,\mu^1)^2,
\end{equation*}
which proves the left bound in \eqref{eq:log-lip-W2}. The other bound follows analogously.
\end{proof}

\begin{definition}
  We say that a curve $(\mu_t)_{t\geq0}$ is a dynamic (upward)
  EVI$^{\ -}$-gradient flow for the entropy if it is locally
  absolutely continuous on $(0,\infty)$ and continuous at $0$ and for
  all $t\in(0,T)$ and all $\sigma\in\cP(X)$ we have
 \begin{align*}
  \frac12\partial_s^-\cW_{s,t}(\mu_s,\sigma)^2_{s=t-}\geq \cH_t(\mu_t)-\cH_t(\sigma)\;.
 \end{align*}
\end{definition}

\begin{proposition}[Gradient estimate implies EVI$^-$-dyn]\label{prop:grad-est-evi}
  Assume that the gradient estimate \eqref{eq:grad-est} holds for on
  the interval $I$ . For $\mu\in\PX$ and $\tau\in I$ let
  $\mu_t=\hat P_{\tau,t}\mu$ denote the dual heat flow starting from
  $\mu$. Then for all $s,t\in I$ with $s\leq t\leq \tau$ and
  $\sigma \in \PX$ we have:
  \begin{align}\label{eq:EVI-long}
    \cH_s(\mu_s)-\cH_t(\sigma)\leq \frac{1}{2(t-s)}
   \Big[\cW_t(\mu_t,\sigma)^2-\cW_{s,t}(\mu_s,\sigma)^2\Big]
  -(t-s)\int_0^1\ip{\dot p_{\theta(a)},\mu^a_\theta}\dd a\;.
  \end{align}
  Here $(\mu^a)_{a\in[0,1]}$ is a $\cW_t$-geodesic connecting
  $\mu^0=\mu_t$ to $\mu^1=\sigma$ and we have put
  $\mu^a_\theta:=\hat P_{t,\theta(a)}\mu^a$. Recall the notation
  $p_s:=\log \pi_s$.

In particular, $\mu_t$ is a dynamic upward EVI$^{\ -}$-gradient flow.
\end{proposition}

For the proof we need the following result.

\begin{proposition}[Action estimate]\label{prop:action-est}
  Assume that the gradient estimate \eqref{eq:grad-est} holds on
  $I$. Fix $s,t,\tau\in I$ with $s\leq t\leq \tau$, let
  $(\mu,V)\in\CE_1(\mu^0,\mu^1)$ be such that $a\mapsto \mu^a$ is
  $C^1$, and let $\varphi\in\HJ_{s,t}$. Moreover, put
  $\mu^a_\theta:=\hat P_{\tau,\theta(a)}\mu^a$. Then we have
  \begin{align}\nonumber
    &\ip{\phi^1,\mu^1_\theta}-\ip{\phi^0,\mu^0_\theta} -\int_0^1\frac12\cA_\tau(\mu^a,V^a)\dd a\\\label{eq:action-est}
  &\leq
  (t-s)\big[\cH_t(\mu^1_\theta)- \cH_s(\mu^0_\theta)\big]
   -(t-s)^2\int_0^1\ip{\dot p_{\theta(a)},\mu^a_\theta}\dd a\;.
  \end{align}
\end{proposition}

Recall the shorthand notation $\norm{\nabla\psi}_{\mu,r}^2=\Gamma_r(\mu,\psi)$.

\begin{proof}
  Let us put $g^a_\theta:=\log \rho^a_\theta$, where
  $\mu^a_\theta=\rho^a_\theta\pi_{\theta(a)}$. We first calculate
  \begin{align}\nonumber
    \frac{\dd}{\dd a}\ip{\phi^a,\mu^a_\theta}
    &=
    \ip{\dot \phi^a,\mu^a_\theta}
    +(t-s)\ip{\phi^a,-\hat\Delta_{\theta(a)}\mu^a_\theta}
    +\ip{ P_{\tau,\theta(a)}\phi^a,\dot\mu^a}\\\label{eq:Action-est1}
    &\leq
    -\frac12\abs{\nabla\phi^a}^2_{\mu^a_\theta,\theta(a)}
    +(t-s)\ip{\nabla\phi^a,\nabla g^a_\theta}_{\mu^a_\theta,\theta(a)}
    +\ip{P_{\tau,\theta(a)}\phi^a,\dot\mu^a}\\\nonumber
    &=:I_1\;.
   \end{align}

   Here, we have used that $\phi$ is a HJ-subsolution and the fact
   that for any $\phi\in \R^\cX$ and $\mu=\rho\pi_r\in\PX$ we have
   that $\ip{\phi,\hat\Delta_r\mu} =-\ip{\nabla \phi, \nabla \log\rho
   }_{\mu,r}$.  

    Next, we calculate
   \begin{align}\nonumber
     \frac{\dd}{\dd a}\cH_{\theta(a)}(\mu^a_\theta)
     &=
     \frac{\dd}{\dd a}\sum_{x}\log\frac{\mu^a_\theta(x)}{\pi_{\theta(a)}(x)}\mu^a_\theta(x)\\\nonumber
   &=
     \ip{\log \mu^a_\theta-\log\pi_{\theta(a)},\frac{\dd}{\dd a}\mu^a_\theta} 
  - \ip{\frac{\dd}{\dd a}\log\pi_{\theta(a)},\mu^a_\theta}\\\nonumber
  &= 
  (t-s)\ip{g^a_\theta,-\hat\Delta_{\theta(a)}\mu^a_\theta}
   + \ip{P_{\tau,\theta(a)} g^a_\theta,\dot \mu^a}
   -(t-s)\ip{\dot p_{\theta(a)},\mu^a_\theta}\\\label{eq:Action-est2}
 &= 
  (t-s)\abs{\nabla g^a_\theta}^2_{\mu^a_\theta,\theta(a)}
   + \ip{P_{\tau,\theta(a)} g^a_\theta,\dot \mu^a}
   -(t-s)\ip{\dot p_{\theta(a)},\mu^a_\theta}\\\nonumber
&=: I_2 \;.
   \end{align}
If we set $f^a:=\phi^a+(t-s)g^a_\theta$ we can estimate further
\begin{align*}
  I_1+(t-s)\cdot I_2 
&=
-\frac12\abs{\nabla\phi^a}^2_{\mu^a_\theta,\theta(a)}
    +(t-s)\ip{\nabla g^a_\theta,\nabla f^a}_{\mu^a_\theta,\theta(a)}\\
    &\quad +\ip{P_{\tau,\theta(a)} f^a,\dot \mu^a}
-(t-s)^2\ip{\dot p_{\theta(a)},\mu^a_\theta}\\
&\leq 
\frac12 \cA_\tau(\mu^a,V^a)
-(t-s)^2\ip{\dot p_{\theta(a)},\mu^a_\theta}\\
&\quad
+\frac12\abs{\nabla P_{\tau,\theta(a)} f^a}^2_{\mu^a,\tau}
-\frac12\abs{\nabla\phi^a}^2_{\mu^a_\theta,\theta(a)}
    +(t-s)\ip{\nabla g^a_\theta,\nabla  f^a}_{\mu^a_\theta,\theta(a)}\\
&\leq
\frac12 \cA_\tau(\mu^a,V^a) 
-(t-s)^2\ip{\dot p_{\theta(a)},\mu^a_\theta}\\
&\quad
+\frac12\abs{\nabla f^a}^2_{\mu^a_\theta,\theta(a)}
-\frac12\abs{\nabla \phi^a}^2_{\mu^a_\theta,\theta(a)}
    +(t-s)\ip{\nabla g^a_\theta,\nabla f^a}_{\mu^a_\theta,\theta(a)}\\
&\leq 
\frac12 \cA_\tau(\mu^a,V^a) 
-(t-s)^2\ip{\dot p_{\theta(a)},\mu^a_\theta}\;.
\end{align*}
Here, we have used the gradient estimate in the second inequality and
in the first inequality the fact that for every $r$ and $\psi\in\R^\cX$ we have
\begin{align*}
  \ip{\psi,\dot\mu^a}\leq \cA_r(\mu^a,V^a)+\frac12\norm{\nabla \psi}^2_{\mu^a, r}\;,
\end{align*}
where $V$ is such that $(\mu,V)\in\CE_1$.

Now, the claim follows immediately by integrating the last estimate in
$a$ from $0$ to $1$.
\end{proof}

\begin{proof}[Proof of Proposition \ref{prop:grad-est-evi}]
  By \cite[Lemma~2.9]{EM12} we can find a sequence of $C^1$ curves
  $(\mu_n,V_n)\in\CE_1(\mu_t,\sigma)$ such that
  $\lim_n \int_0^1\cA_t(\mu^a_n,V^a_n)\dd
  a=\cW_t(\mu_t,\sigma)^2$.
  Now we can apply Proposition \ref{prop:action-est} with $\tau=t$ to the curves
  $(\mu^a_n)_{a\in[0,1]}$ and take the limit in $n$ and the
  supremum over HJ-subsolutions $(\phi^a)_a$ using \eqref{eq:dual-W-dyn}.
\end{proof}

\begin{proof}[Proof of $(II)\Rightarrow(IV)$]~\\
  Fix $t\in I$ and let $(\mu^a)_{a\in[0,t]}$ be a
  $\cW_t$-geodesic. From the estimate \eqref{eq:EVI-long} applied to
  $\tau=t$ and $\mu=\mu^a$, $\sigma=\mu^0$ we obtain for $s<t$,
  setting $\mu^a_s=\hat P_{t,s}\mu^a$:
  \begin{align}\nonumber
    \cH_t(\mu^0)-\cH_s(\mu^a_s) &\geq \frac{1}{2(t-s)}\Big[\cW_{s,t}(\mu^a_s,\mu^0)^2-\cW_t(\mu^a,\mu^0)^2\Big]-(t-s)L\\\label{eq:est-left}
                               &\geq \frac{1}{2(t-s)}\Big[\cW_{s}(\mu^a_s,\mu^0)^2-\cW_t(\mu^a,\mu^0)^2\Big] -\frac12a^2L -(t-s)L\;.
  \end{align}
  Here, we have used that $|\dot p|\leq L$ and the control
  \eqref{eq:log-lip-W2}. Similarly, choosing $\mu=\mu^{1-a}$,
  $\sigma=\mu^1$ we obtain:
  \begin{align}\label{eq:est-right}
    \cH_t(\mu^1)-\cH_s(\mu^{1-a}_s)&\geq \frac{1}{2(t-s)}\Big[\cW_{s}(\mu^{1-a}_s,\mu^1)^2-\cW_t(\mu^{1-a},\mu^1)^2\Big] -\frac12a^2L -(t-s)L\;.
  \end{align}
Moreover, the contraction estimate yields 
\begin{align}\label{eq:est-contract}
\cW_s(\mu^a_s,\mu^{1-a}_s)^2 \leq \cW_t(\mu^{a},\mu^{1-a})^2\;.
\end{align}
Adding \eqref{eq:est-left} and \eqref{eq:est-right} multiplied by
$1/a$ and \eqref{eq:est-contract} multiplied by $1/(1-2a)$ we obtain
\begin{align*}
  &\frac1{a}\big[ \cH_t(\mu^0)-\cH_s(\mu^a_s)+\cH_t(\mu^1)-\cH_s(\mu^{1-a}_s)\big]\\
&\geq 
\frac{1}{2(t-s)}\Big[\frac1a\cW_{s}(\mu^0,\mu^a_s)^2 +\frac{1}{1-2a}\cW_s(\mu^a_s,\mu^{1-a}_s)^2 + \frac1a\cW_{s}(\mu^{1-a}_s,\mu^1)^2\\
              &\qquad\quad\qquad-\frac1a\cW_t(\mu^0,\mu^a)^2-\frac1{1-2a}\cW_t(\mu^a,\mu^{1-a})^2-\frac1a\cW_t(\mu^{1-a},\mu^1)^2\Big]\\
              &\qquad  -aL -(t-s)\frac{2L}{a}\\
&\geq
\frac{1}{2(t-s)}\Big[\cW_{s}(\mu^0,\mu^1)^2 -\cW_t(\mu^{0},\mu^1)^2\Big]-aL -(t-s)\frac{2L}{a}\;.
\end{align*}
Now, taking first the $\limsup$ as $s\nearrow t$ and then the $\limsup$ as $a\searrow 0$ yields \eqref{eq:convex2}.
\end{proof}

\subsection{Reverse Poincar\'e inequality for super Ricci flows}\label{sec:FI}
We finish this section by showing that a reverse Poincar\'e inequality
holds on discrete super Ricci flows. A similar result is expected for
super Ricci flows of metric measure spaces and is currently
investigation \cite{KS18+}. In fact it is expected that local
Poincar\'e inequalities and other Harnack type inequalities can be
used to characterize super Ricci flows in the continuous setting.

\begin{theorem}[Reverse Poincar\'e inequality]\label{thm:loc-Poinc}
  Let $(X_t,Q_t,\pi_t)_{t\in[0,T]}$ be a super-Ricci flow. Then the
  one-sided local Poincar\'e inequality holds, i.e.~for all $s\leq t$
  and all $\mu\in\cP(\cX_t)$, $\psi\in\R^{\cX_s}$ we have
\begin{align}\label{eq:loc-Poinc}
\ip{P_{t,s}(\psi^2),\mu}-\ip{(P_{t,s}\psi)^2,\mu}\geq 2(t-s)\Gamma_t(\mu,P_{t,s}\psi)\;.
\end{align}
\end{theorem}
\begin{proof}
Define for $s\leq r\leq t$ the function $h(r)=\ip{(P_{r,s}\psi)^2,\hat P_{t,r}\mu}$. Then for a.e.~$r\in(s,t)$
\begin{align*}
h'(r)=-\ip{\Delta_r(P_{r,s}\psi)^2,\hat P_{t,r}\mu}+2\ip{P_{r,s}\psi\Delta_rP_{r,s}\psi,\hat P_{t,r}\mu}\;.
\end{align*}
Note that for the Laplacian satisfies for all $\psi\in\R^{\cX_r}$
\begin{align*}
\Delta_r\psi^2(x)=2\psi(x)\Delta_r\psi(x)+\sum_{y\in\cX_r}|\nabla\psi|^2(x,y)Q_r(x,y)\;.
\end{align*}
 Consequently we have
 \begin{align*}
 h'(r)=&-\sum_{x,y\in\cX}|\nabla P_{r,s}\psi|^2(x,y)Q_r(x,y)\hat P_{t,r}\mu(x)\\
 =&-\sum_{x,y\in\cX}|\nabla P_{r,s}\psi|^2(x,y)\frac{Q_r(x,y)\hat P_{t,r}\mu(x)+Q_r(y,x)\hat P_{t,r}\mu(y)}2\\
 \leq &-2\Gamma_r(\hat P_{t,r}\mu,P_{r,s}\psi)
 \end{align*}
 where we used the reversibility of the chain and that the logarithmic
 mean is dominated by the arithmetic mean,
 i.e.~$\Lambda(s,t)\leq (s+t)/2$. The gradient estimate readily
 implies
 \begin{align*}
 h'(r)\leq-2\Gamma_t(\mu,P_{t,s}\psi).
 \end{align*}
 Noting that $r\mapsto h(r)$ is continuous on $[r,s]$ we can integrate the last estimate to prove the claim.
\end{proof}

\section{Examples}
\label{sec:ex}

A first elementary example of super Ricci flows are static Markov
triples with non-negative Ricci curvature.

\begin{example}\label{ex:static-soliton}
  Let $(\cX,Q,\pi)$ be a Markov triple with $\Ric(\cX,Q,\pi)\geq 0$
  and let $Q_t=Q$ and $\pi_t=\pi$ for all $t$. Then $(\cX,Q_t,\pi_t)$
  is a super Ricci flow. Indeed, by Proposition \ref{prop:Ric-equiv} we have that $\Gamma_{2,t}(\mu,\psi)\geq 0$ for all
  $\mu\in\cP_*(\cX),\psi\in\R^\cX$ and obviously we have $\partial_t\Gamma_t(\mu,\psi)=0$.
\end{example}

More generally, any homogeneous Markov triple with a positive (negative) lower
Ricci bound gives rise to a shrinking (expanding) soliton-like super
Ricci flow.

\begin{example}\label{ex:shrink-exp-solition}
  Let $(\cX,Q,\pi)$ be a Markov triple with $\Ric(\cX,Q,\pi)\geq \kappa$ for some $\kappa\in\R$. Define
  \begin{align*}
  L_t=\frac1{1-2\kappa Rt}
  \end{align*}
  and put $Q_t=L_tQ$ and $\pi_t=\pi$ for $t\in I$ with
  $I=[0,1/2\kappa R)$ if $\kappa>0$ or $I=[0,\infty)$ if
  $\kappa\leq 0$. Then $(\cX,Q_t,\pi_t)_{t\in I}$ is a super Ricci
  flow. Indeed, for all $\mu\in\cP_*(\cX)$, $\psi\in\R^\cX$ we have
  that:
  \begin{align*}
    \Gamma_{2,t}(\mu,\psi)=L_t^2\cdot\Gamma_{2,0}(\mu,\psi) \geq L_t^2 \kappa\cdot \Gamma_0(\mu,\psi)
    = L_t\kappa\cdot\Gamma_t(\mu,\psi)\;.  
  \end{align*}
  Moreover, we have
  \begin{align*}
  \partial_t\Gamma_t(\mu,\psi)=\dot L_t\cdot \Gamma_0(\mu,\psi)
   = \frac{\dot L_t}{L_t}\cdot \Gamma_t(\mu,\psi)\;.
  \end{align*}
  Since $L_t$ satisfies the ODE $\dot L_t = 2\kappa L_t^2$, we have
  $\Gamma_{2,t}(\mu,\psi)\geq \frac12 \partial_t\Gamma_t(\mu,\psi)$ as
  required.
\end{example}

We can interpret growing transition rates as a \emph{shrinking} of the
corresponding graph and decreasing rates as an \emph{expansion}. Thus
in the case $\kappa>0$ the super Ricci flow collapses to a point at
time $t_1=1/2\kappa R$.

We can combine these effects to produce examples of flows evolving
across singular times featuring \emph{collapse} and \emph{explosion}
of vertices. To this end we recall the notion of product of two Markov
triples. Given Markov triples $(\cX^1,Q^1,\pi^1)$, $(\cX^2,Q^2,\pi^2)$
we denote by
$(\cX^1,Q^1,\pi^1)\otimes(\cX^2,Q^2,\pi^2):=(\cX^1\times\cX^2,Q^1\otimes
Q^2,\pi^1\otimes\pi^2)$
the Markov chain evolving on the product space taking independent
jumps in each factor, i.e.~for distinct pairs
$(x^1,x^2),(y^1,y^2)\in\cX^1\times\cX^2$ we set
\begin{align*}
  Q^1\otimes Q^2\big((x^1,x^2),(y^1,y^2)\big) :=
  \begin{cases}
    Q^1(x^1,y^1)\;, & x^2=y^2\;,\\
    Q^2(x^2,y^2)\;, & x^1=y^1\;,\\
    0\;, & \text{else}\;.
  \end{cases}
\end{align*}
This chain is reversible w.r.t.~the product measure $\pi^1\otimes\pi^2$.
  \begin{example}\label{ex:collapse}
 Let $(\cY,Q^Y,\pi^Y), (\cZ,Q^Z,\pi^Z)$ be Markov triples with $\text{Ric}(\cY)\geq0$, $\text{Ric}(\cZ)\geq\kappa>0$. Then the time-dependent triple
 \begin{align*}
   (\cX_t,Q_t,\pi_t):=
   \begin{cases}
    (\cY,Q^Y,\pi^Y)\otimes(\cZ,L_tQ^Z,\pi^Z)\;, & 0\leq t < t_1:=1/2\kappa\;,\\
      (\cY,Q^Y,\pi^Y)\;, & t \geq t_1\;,\\
   \end{cases}
   \end{align*}
   with $L_t=1/(1-2\kappa t)$ is a super Ricci flow.

   Indeed, it is readily checked that this choice of rates satisfies
   the condition in Sections \ref{sec:topo} and \ref{sec:RFequiv}. In
   view of Theorem \ref{thm:srf-equiv} we only need to check that
   Bochner's inequality is satisfied for a.e.~$t$. This follows from
   the same argument as in the previous examples together with the
   fact that for $t<t_1$ we have (see the e.g.~the proof of
   tensorization principle \cite[Thm.~6.2]{EM12}):
   \begin{align*}
 \Gamma_{2,t}(\mu,\psi)\geq\sum_{z\in\cZ}\Gamma_{2}^\cY\big(\mu(\cdot,z),\psi(\cdot,z)\big) + L_t^2\sum_{y\in\cY}\Gamma_{2}^\cZ\big(\mu(y,\cdot),\psi(y,\cdot)\big)\;,
   \end{align*}
   where $\Gamma_2^\cY,\Gamma_2^\cZ$ denote the integrated carr\'e du
   champs operators calculated for the triples on $\cY$ and $\cZ$
   respectively.
\end{example}

In the previous example, at the singular time $t_1$ a positively
curved factor collapses to a point. Similarly, we can consider an
evolution where at a singular time each vertex explodes into a chain
with negative Ricci bound.

\begin{example}\label{ex:explosion}
  Let $(\cY,Q^Y,\pi^Y), (\cZ,Q^Z,\pi^Z)$ be Markov triples with
  $\text{Ric}(\cY)\geq0$, $\text{Ric}(\cZ)\geq\kappa$ for
  $\kappa<0$. Then the time-dependent triple
 \begin{align*}
   (\cX_t,Q_t,\pi_t):=
   \begin{cases}
   (\cY,Q^Y,\pi^Y)\;, & 0\leq t \leq t_1\;,\\  
   (\cY,Q^Y,\pi^Y)\otimes(\cZ,L_tQ^Z,\pi^Z)\;, & t\geq t_1\;,\\
   \end{cases}
   \end{align*}
  with $L_t=-1/2\kappa (t-t_1)$ is a super Ricci flow.
\end{example}

\begin{example}
  Consider the time-dependent two-point space
  $(\{a,b\}, Q_t,\pi)_{t\in (0,T)}$ and assume $Q_t(a,b)=Q_t(b,a)=p_t$
  with $p_t=\frac1{1-4tp_0}p_0$. This is a super Ricci flow up to
  collapsing time $T=\frac1{4p_0}$ as we have seen in Example
  \ref{ex:shrink-exp-solition} and it is optimal since
\begin{align*}
\Gamma_{2,t}(\mu,\psi)\geq 2p_t\Gamma_t(\mu,\psi) = \frac12\partial_t\Gamma_t(\mu,\psi),
\end{align*}
and $2p_t$ is the optimal lower Ricci bound for each
$(\{a,b\},Q_t,\pi)$, i.e.~the optimal constant in the first
inequality, see \cite[Prop.2.12]{Ma11}.
  
\end{example}

\section{Stability of super Ricci flows}
\label{sec:stability}

In this section we show that our notion of discrete super Ricci flow
is consistent with classical super Ricci flows on manifolds, and more
generally the synthetic notion considered in \cite{sturm2015}, in a
discrete to continuum limit. More precisely, we show that if a
sequence of discrete super Ricci flows (with some uniform control on
the distances) converges to a time-dependent continuous metric measure
space in a suitable weak sense then the latter is a super Ricci flow
in the sense of \cite{sturm2015}.

Let us first recall the definitions. A time-dependent metric measure
space is a family\\
$(X,d_t,m_t)_{t\in I}$ for an (left open) interval $I\subset \R$,
where $X$ is a compact Polish space and for each $t$, $m_t$ is a Borel
probability measure on $X$ and $d_t$ is a geodesic metric on $X$
generating the given topology. One also assumes that all measures
$m_t$ are absolutely continuous w.r.t.~each other, more precisely
there exists a bounded measurable function $f:I\times X\to\R$ and a
probability measure $m$ such that $m_t=e^{-f_t}m$ for all $t\in I$. We
denote by $\cH_t(\mu):=\ent(\mu|m_t)$ the Boltzmann entropy of
$\mu\in\cP(X)$ relative to $m_t$ given by
$\cH_t(\mu)=\int\rho\log\rho\dd m_t$, provided $\mu=\rho m_t$ and
$+\infty$ else. Note that $\cH_t(\mu)=\ent(\mu|m)+\int f_t\dd\mu$,
thus in particular the condition $\cH_t(\mu)<\infty$ is independent of
$t$. We denote by $W_{2,t}$ the $L^2$-Kantorovich distance associated
to $d_t$, i.e.~for $\mu,\nu\in\cP(X)$
\begin{align*}
  W_{2,t}^2(\mu,\nu) = \inf_q \int d_t(x,y)^2 \dd q(x,y)\;,
\end{align*}
where the infimum runs over all couplings of $\mu$ and $\nu$.

\begin{definition}[{\cite[Def.~2.4]{sturm2015}}]\label{def:SRF-cont}
  A time-dependent mm-space $(X,d_t,m_t)_{t\in I}$ is a super Ricci
  flow if the Boltzmann entropy is dynamically convex, i.e.: for
  a.e.~$t\in I$ and every $\mu^0,\mu^1\in\cP(X)$ there exists a
  $W_{2,t}$-geodesic $(\mu^a)_{a\in[0,1]}$ connecting $\mu^0$ to
  $\mu^1$ such that $a\mapsto \ent(\mu^a|m_t)$ is absolutely
  continuous on $[0,1]$ and
 \begin{align}\label{eq:SRF-cont}
   \partial_a^+\cH_t(\mu^{1-})-\partial_a^-\cH_t(\mu^{0+}) \geq -\frac12\partial_t^-W_{2,t}^2(\mu^0,\mu^1)\;.
 \end{align}
\end{definition}

We now introduce a suitable notion of convergence of a sequence of
time-dependent Markov triples to a time-dependent continuous mm-space.

\begin{definition}\label{def:disc-cont-conv}
  A sequence $(\cX^{(n)},Q^{(n)}_t,\pi^{(n)}_t)_{t\in I}$ of
  time-dependent Markov triples converges to a time-dependent mm space
  $(X,d_t,m_t)_{t\in I}$ if there exist maps
  $i_n:\cP(\cX^{(n)}) \to \cP(X)$ such that:
  \begin{itemize}
  \item[(i)] for each $J=(r,s)\subset I$ and for each family of
    sequences $\mu_t^{n,0},\mu_t^{n,1}\in \cP(\cX^{(n)})$ for $t\in J$
    such that $i_n(\mu_t^{n,j})\dd t\to\mu_t^j\dd t$ weakly as
    measures on $X\times [r,s]$ for $j=0,1$ and some families
    $\mu_t^0,\mu_t^1\in\cP(X)$:
      \begin{align*}
        \int_J\cH_t(\mu_t^j)\dd t &\leq \liminf_{n\to\infty} \int_J\cH_t^{(n)}\big(\mu_t^{n,j}\big)\dd t\;,\\
        \int_J W_{2,t}(\mu_t^0,\mu_t^1)^2\dd t &\leq \liminf_{n\to\infty}\int _J\cW^{(n)}_t(\mu_t^{n,0},\mu_t^{n,1})^2\dd t\;,
      \end{align*}
    \item[(ii)] for each $J=(r,s)\subset I$ and for each
      $\mu^0,\mu^1\in \cP(X)$ there exist sequences
      $\mu^{n,j}\in \cP(\cX^{(n)})$ such that for $j=0,1$ we have
      $i_n(\mu^{n,j})\to\mu^j$ weakly and:
      \begin{align*}
        \int_J\cH_t(\mu^j)\dd t &= \lim_{n\to\infty} \int_J\cH_t^{(n)}\big(\mu^{n,j}\big)\dd t\;,\\
W_{2,t}(\mu^0,\mu^1) &= \lim_{n\to\infty}\cW^{(n)}_t(\mu^{n,0},\mu^{n,1}) \text{ for a.e. } t\in J\;.
      \end{align*}
    \end{itemize}
 \end{definition}

\begin{remark}
 The intuition behind is that we think of $\cX^{(n)}$ as finer and
 finer discretizations of $X$ and the Markov generators $\Delta^{(n)}$
 associated to $Q^{(n)}$ as discretizations of the canonical Laplacian
 on $(X,d,m)$. For instance, $\cX^{(n)}$ could be the set of vertices
 of a mesh in $X$ and the map $i_n$ a suitable extension of a measure
 on $\cX^{(n)}$ to a measure on $X$ via interpolation or convolution
 with a mollifying kernel.

 In this scenario, we might expect additional properties of the maps
 $i_n$ that allow to verify the assumptions of Definition
 \ref{def:disc-cont-conv}. Motivated by the results on
 Gromov--Hausdorff convergence of discrete transport distances in
 \cite{GM12, Tri, GKM}, we could expect that $i_n$ are approximate
 isometries, i.e.
 \begin{align*}
\big|W_{2,t}(i_n(\mu^{0}),i_n(\mu^{1}))-\cW_t^{(n)}(\mu^{0},\mu^{1})\big|\leq\eps_n
 \end{align*}
 for all $\mu^0,\mu^1\in\cP(\cX^{(n)})$ and a.e.~$t\in I$ and for all $\mu\in\cP(X)$ there exist $\mu^n\in\cP(\cX^{(n)})$ such that 
\begin{align*}
& W_{2,t}(i_n(\mu^{n}),\mu)\leq \eps_n\;,
  \end{align*}
for a.e.~$t\in I$, where $\eps_n\to0$ as $n\to\infty$.
 \end{remark}

 Note that by our assumptions in Section \ref{sec:topo}, if
 $(\cX,Q_t,\pi_t)_{t\in I}$ is a super Ricci flow with constant base
 space, then $t\mapsto\pi_t(x)$ is in particular continuous and
 bounded away from $0$ and $\infty$. Thus there exists a bounded
 continuous $f:I\times\cX\to\R$ with $\pi_t=e^{-f_t}\pi_{t_*}$ for
 some $t_*\in I$.
 
 We need the following additional notion of control on the time
 regularity of the flows: We say that a time-dependent Markov triple
 $(X,Q_t,\pi_t)_{I}$ is \emph{moderate} if there exists a function
 $t\mapsto \lambda_t$ in $L^1_{\text{loc}}(I)$ such that
 \begin{align}\label{eq:lambda-mod}
   Q_t(x,y)\geq L_{t,s}Q_s(x,y)\;,\quad\forall s\leq t\;,\ x,y\in\cX\;,
 \end{align}
 where
 \begin{align*}
  L_{t,s}:= \exp\Big(-\int_s^t\lambda_r\Big)\;.
 \end{align*}
 We call $\lambda$ the \emph{control function} in this case.

 Note that the control \eqref{eq:lambda-mod} on the rates immediately
 implies the control $\cW_t(\mu,\nu)^2 \geq L_{t,s}\cW_s(\mu,\nu)^2$
 on the transportation costs for all $s\leq t$ and $\mu\nu\in\cP(\cX)$
 and in turn
 \begin{align}\label{eq:lambda-mod-W}
\partial_t^-\cW_{t-}(\mu,\nu)^2\geq -\lambda_t\cW_t(\mu,\nu)^2 \;,\quad \text{a.e.}\ t\;,\ \mu,\nu\in\cP(\cX)\;.
 \end{align}
  
  We have the following stability result for discrete super Ricci flows.

 \begin{theorem}\label{thm:stable} 
   Let $(\cX^{(n)}, Q_t^{(n)},\pi^{(n)}_t)_{t\in I}$ be a sequence of
   moderate super Ricci flows with control function $\lambda$ and such
   that $\text{diam}\big(\cP(\cX^{(n)}),\cW^{(n)}_t\big)\leq L$ for all $n$
   and $t\in I$, which converges to a time-dependent mm-space
   $(X,d_t,m_t)_{t\in I}$. Then $(X,d_t,m_t)_{t\in I}$ is a super Ricci flow.
 \end{theorem}

 The proof of stability follows from the fact that under the control
 \eqref{eq:lambda-mod} the dynamic convexity property can be
 reformulated in an integrated way, following the reasoning in
 \cite[Thm.~3.3]{sturm2015} for stability of super Ricci flows of
 mm-spaces, see in particular \cite[Thm.~1.15,
 Prop.~2.21]{sturm2015}. For the reader's convenience we recapitulate
 this in the present setting in the first and the last step of the
 proof.

 \begin{proof}
   {\bf Step 1:} By assumption (c.f.~Theorem \ref{thm:srf-equiv}) we
   have dynamic convexity of the entropy
   $\cH^{(n)}_t=\ent(\cdot|\pi^{(n)}_t)$ on
   $(\cX^{(n)},Q_t^{(n)},\pi_t^{(n)})_t$. I.e.~for a.e.~$t\in I$,
   every $n$, and every $\cW^{(n)}_t$-geodesic $\mu^{n,a}$ such that
   \begin{align}\label{eq:convex1}
      \partial_a^+\cH^{(n)}_t(\mu^{n,1-})-\partial_a^-\cH^{(n)}_t(\mu^{n,0+})\geq-\frac12\partial_t^-\cW^{(n)}_{t-}(\mu^{n,0},\mu^{n,1})^2\;.
   \end{align}
   We will first pass to an integrated version of \eqref{eq:convex1}
   in space and time. More precisely, we claim: for every $n$, every
   $J=(r,s)\subset I$ and every measurable family of
   $\cW^{(n)}_t$-geodesics $(\mu^{n,a}_t)_{a\in[0,1]}$ connecting
   $\mu^{n,0},\mu^{n,1}$ for $t\in J$ and every $\tau\in(0,\frac12)$
   we have that
   \begin{align}\nonumber
    & \cH_J^{(n)}(\mu^{n,1})- \cH_J^{(n)}(\mu^{n,1-\tau}_J)+ \cH_J^{(n)}(\mu^{n,0})- \cH_J^{(n)}(\mu^{n,\tau}_J)\\\label{eq:convex2}
 &\geq -\frac{\tau}{2(s-r)}\Big[\cW^{(n)}_{s}(\mu^{n,0},\mu^{n,1})^2-\cW^{(n)}_{r}(\mu^{n,0},\mu^{n,1})^2\Big] 
-\tau^2\cW^{(n)}_{\lambda J}(\mu^{n,0},\mu^{n,1})^2\;,
   \end{align}
   where we have put
   \begin{align*}
     \cH^{(n)}_J(\mu^{n,a}_J)&:=\frac{1}{s-r}\int_r^s\cH^{(n)}_t(\mu^{n,a}_t)\dd t\;,\\
   \cW^{(n)}_{\lambda J}(\mu^{n,0},\mu^{n,1})^2&:= \frac{1}{s-r}\int_r^s \lambda_t\cW^{(n)}_t(\mu^{n,0},\mu^{n,1})^2\dd t\;.
   \end{align*}

   To show this, first note that for all $t$ we have a curvature bound
   $\Ric(\cX^{(n)},Q_t^{(n)},\pi_t^{(n)})\geq \kappa_t$ for some
   $\kappa_t\in\R$ by (\cite[Theorem 4.1]{Mie11b}) implying that
   $\cH^{(n)}_t$ is semiconvex along $\cW^{(n)}_t$-geodesics. Thus for
   the geodesics
   $a\mapsto\mu^{n,a}_{t,\sigma}:=\mu_t^{n,\sigma+a(1-2\sigma)}$ we
   can write
   \begin{align*}
   \cH_t^{(n)}(\mu^{n,\tau}_t)-\cH_t^{(n)}(\mu^{n,0})= \int_0^\tau\partial_\sigma\cH^{(n)}_t(\mu^{n,\sigma}_t)\dd\sigma = \int_0^\tau\frac{1}{1-2\sigma}\partial_a\cH^{(n)}_t(\mu^{n,a}_{t,\sigma})\big|_{a=0}\dd\sigma\;,
   \end{align*}
   and
   $\cH_t^{(n)}(\mu^{n,1})-\cH_t^{(n)}(\mu^{n,1-\tau}_t)=
   \int_0^\tau\frac{1}{1-2\sigma}\partial_a\cH^{(n)}_t(\mu^{n,a}_{t,\sigma})\big|_{a=1}\dd\sigma$.
   Adding these identities and using \eqref{eq:convex1} for the
   geodesics $\mu^{n,a}_{t,\sigma}$ yields
\begin{align*}
&\cH_t^{(n)}(\mu^{n,1})-\cH_t^{(n)}(\mu^{n,1-\tau}_t)+   \cH_t^{(n)}(\mu^{n,0})-\cH_t^{(n)}(\mu^{n,\tau}_t)\geq 
-\frac12\int_0^\tau\frac{1}{1-2\sigma}\partial_t^-\cW^{(n)}_{t-}(\mu^{n,\sigma}_t,\mu^{n,1-\sigma}_t)^2\dd\sigma\;.
\end{align*}
Noting that for the $\cW^{(n)}_t$-geodesic $\mu^{n,\cdot}_t$ we have
\begin{align*}
  \partial_t^-\cW^{(n)}_t(\mu^{n,0},\mu^{n,1})^2
 &\geq \frac{1}{1-2\sigma}\partial_t^-\cW^{(n)}_t(\mu^{n,\sigma}_t,\mu^{n,1-\sigma}_t)^2+\frac1\sigma\partial_t^-\cW^{(n)}_t(\mu^{n,0},\mu^{n,\sigma}_t)^2\\ &+\frac1\sigma\partial_t^-\cW^{(n)}_t(\mu^{n,1-\sigma}_t,\mu^{n,1})^2\;,
\end{align*}
and the bound \eqref{eq:lambda-mod-W} and integrating in $\sigma$ then yields
\begin{align}\nonumber
  &\cH_t^{(n)}(\mu^{n,1})-\cH_t^{(n)}(\mu^{n,1-\tau}_t)+   \cH_t^{(n)}(\mu^{n,0})-\cH_t^{(n)}(\mu^{n,\tau}_t)\\\label{eq:convex3} &\geq 
-\frac{\tau}2\partial_{t}^-\cW^{(n)}_{t-}(\mu^{n,0},\mu^{n,1})^2
-\tau^2\lambda_t\cW^{(n)}_t(\mu^{n,0},\mu^{n,1})^2\;.
\end{align}
To obtain \eqref{eq:convex2} it suffices to integrate
\eqref{eq:convex3} in $t$ on $(r,s)$ using that
$\cW^{(n)}_s(\mu,\nu)^2-\cW^{(n)}_r(\mu,\nu)^2\geq
\int_r^s\partial_t^-\cW^{(n)}_{t-}(\mu,\nu)^2\dd t$.
This last estimate can be deduced from the fact that that
$\cW^{(n)}_s(\mu,\nu)^2-\cW^{(n)}_r(\mu,\nu)^2\geq \int_r^s\eta_t\dd
t$
for some $\eta\in L^1_{\text{loc}}$, more precisely by the lower
log-Lipschitz bound we can take
$\eta_t=\exp(\int_s^t\lambda_u\dd u)\lambda_uL^2$, where $L$ is a
uniform bound on $\cW^{(n)}$.

\medskip {\bf Step 2:} Now, we pass to the limit in \eqref{eq:convex2}
as $n\to\infty$. We claim that for every $J=(r,s)\subset I$ and every
$\mu^{0},\mu^{1}\in\cP(X)$ there exists a measurable family of
$W_{2,t}$-geodesics $(\mu^{a}_t)_a$ connecting $\mu^{0},\mu^{1}$ for
$t\in J$ such that for every $\tau\in(0,\frac12)$ we have that
   \begin{align}\nonumber
    & \cH_J(\mu^{1})- \cH_J(\mu^{1-\tau}_J)+ \cH_J(\mu^{0})- \cH_J(\mu^{\tau}_J)\\\label{eq:convex4}
 &\geq -\frac{\tau}{2(s-r)}\Big[W_{2,s}(\mu^{0},\mu^{1})^2-W_{2,r}(\mu^{0},\mu^{1})^2\Big] 
-\tau^2 W_{2,\lambda J}(\mu^0,\mu^{1})^2\;,
   \end{align}
   where we have put again
   \begin{align*}
     \cH_J(\mu^{a}_J)&:=\frac{1}{s-r}\int_r^s\cH_t(\mu^{a}_t)\dd t\;,\quad
   W_{2,\lambda J}(\mu^0,\mu^{1})^2:= \frac{1}{s-r}\int_r^s \lambda_tW_{2,t}(\mu^{0},\mu^{1})^2\dd t\;.
   \end{align*}
   Indeed, by Definition \ref{def:disc-cont-conv} we can find
   sequences $\mu^{n,0},\mu^{n,1}\in\cP(\cX^{(n)})$ such that
   $i_n(\mu^{n,j})\to \mu^j$ weakly and
   $\cH^{(n)}_J(\mu^{n,j})\to\cH_J(\mu^j)$ for $j=0,1$ as $n\to\infty$
   as well as
   $\cW^{(n)}_t(\mu^{n,0},\mu^{n,1})\to W_{2,t}(\mu^0,\mu^1)$ for
   a.e.~$t\in J$. By the uniform bound on $\cW^{(n)}_t$ this implies
   also that
   $\cW^{(n)}_{\lambda J}(\mu^{n,0},\mu^{n,1})\to W_{2,\lambda
     J}(\mu^0,\mu^1)$.
   By the previous step there exist a family of
   $\cW^{(n)}_t$-geodesics $(\mu^{n,a}_t)$ for $t\in J$ connecting
   $\mu^{n,0}$ and $\mu^{n,1}$ for which \eqref{eq:convex2} holds. Let
   $\bar \mu^{n,a}_t\in \cP(X)$ be the image of $\mu^{n,a}_t$ under
   $i_n$ and put
   $\bar\mu^{n,a}(\dd x,\dd t)=\bar \mu_t^{n,a}(\dd x)\dd t$. By
   compactness of $X\times \bar J$, we can find measures
   $\mu^a(\dd x,\dd t)$ for $a\in[0,1]\cap\Q$ such that up to
   extracting a subsequence we have that
   $\bar\mu^{n,a}(\dd x,\dd t)\to \mu^a(\dd x,\dd t)$ weakly. It is
   readily checked that the limiting measures take the form
   $\mu^a(\dd x,\dd t)=\mu^a_t(\dd x)\dd t$ for a family of measures
   $\mu_t^a(\dd x)\in\cP(X)$. Again by Definition
   \ref{def:disc-cont-conv} we have for all rational $a,b$:
   \begin{align*}
W_{J}(\mu^a_J,\mu^b_J)^2\leq \liminf_n\cW^{(n)}_J(\mu^{n,a}_J,\mu^{n,b}_J)^2 = \liminf_n(b-a)^2 \cW^{(n)}_J(\mu^{n,0}_J,\mu^{n,1}_J)^2=(b-a)^2W_{J}(\mu^0,\mu^1)^2\;
   \end{align*}
   where we have set
   $W_J(\mu_J^a,\mu_J^b)^2=\frac{1}{|J|}\int_J
   W_{2,t}(\mu_t^a,\mu_t^b)^2\dd t$,
   analogously for $\cW_J$. In particular, we obtain
   \begin{align*}
     \frac1{a}W_{J}(\mu^0,\mu_J^a)^2+\frac1{1-a}W_{J}(\mu_J^a,\mu^1)^2\leq W_{J}^2(\mu^0,\mu^1)^2\;.
   \end{align*}
 Note that this implies that for a.e.
   $t\in J$ and all rational $a\in[0,1]$
   \begin{align}\label{eq: geodesic}
   \frac1{a}W_{2,t}(\mu^0,\mu_t^a)^2+\frac1{1-a}W_{2,t}(\mu_t^a,\mu^1)^2= W_{2,t}^2(\mu^0,\mu^1)^2\;,
   \end{align}
   since the ``$\geq$'' in \eqref{eq: geodesic} holds by the triangle
   inequality. Further we entail from Definition
   \ref{def:disc-cont-conv} that for all rational $\tau\in[0,1]$
   \begin{align*}
     \cH_J(\mu^{\tau}_J) \leq \frac1{s-r}\int_r^s\liminf_n \cH^{(n)}_t(\mu^{n,\tau}_t)\dd t \leq  \liminf_n\cH^{(n)}_J(\mu^{n,\tau}_J)\;,
   \end{align*}
   and similarly
   $\cH_J(\mu^{1-\tau}_J)\leq
   \liminf_n\cH^{(n)}_J(\mu^{n,1-\tau}_J)$.
   Thus, we can pass to the limit (inferior) in \eqref{eq:convex2} to
   obtain \eqref{eq:convex4} at all rational $\tau$.
  
   To conclude, we note that \eqref{eq: geodesic} implies that
   $W_{2,t}(\mu^a_t,\mu^b_t)=|b-a|W_{2,t}(\mu^0,\mu^1)$ for almost
   every $t\in J$ and all rational $a,b$. Thus, by completeness of
   $(\cP(X),W_{2,t})$, we can extend for a.e.~$t$ the family
   $(\mu_t^a)_{a\in[0,1]\cap\Q}$ to a $W_{2,t}$-geodesic
   $(\mu_t^a)_{a\in[0,1]}$. By lower semicontinuity of the relative
   entropy and Fatou's Lemma we extend the estimate \eqref{eq:convex4}
   to all $\tau\in (0,\frac12)$.

\medskip

{\bf Step 3:} Finally, we deduce from \eqref{eq:convex4} that
$(X,d_t,m_t)$ is a super Ricci flow in the sense of Definition
\ref{def:SRF-cont}. To do so, note that we can choose a common family
of geodesics for all rational $r\leq s$. Then we let $r,s\to t$ using
Lebesgue's density theorem to obtain for a.e.~$t$ that
\begin{align}\label{eq:convex5}
   \cH_t(\mu^{1})- \cH_t(\mu^{1-\tau}_t)+ \cH_t(\mu^{0})- \cH_t(\mu^{\tau}_t)\geq -\frac{\tau}{2}\partial_t^-W_{2,t-}(\mu^{0},\mu^{1})^2 
-\tau^2 \lambda_tW_{2,t}(\mu^0,\mu^{1})^2\;.
\end{align}
Then it suffices to divide by $\tau$ and let $\tau\to0$ in
\eqref{eq:convex5} to obtain \eqref{eq:SRF-cont}.
 \end{proof}

\bibliographystyle{plain}
\bibliography{dRF}

\begin{thebibliography}{10}

\bibitem{BK17}
R.~Bamler and B.~Kleiner.
\newblock Uniqueness and stability of {R}icci flow through singularities.
\newblock {\em arXiv:1709.04122}, 2017.

\bibitem{BHLLMY15}
F.~Bauer, P.~Horn, Y.~Lin, G.~Lippner, D.~Mangoubi, and S.-T. Yau.
\newblock Li-{Y}au inequality on graphs.
\newblock {\em J. Differential Geom.}, 99(3):359--405, 2015.

\bibitem{BB00}
J.-D. Benamou and Y.~Brenier.
\newblock A computational fluid mechanics solution to the {M}onge-{K}antorovich
  mass transfer problem.
\newblock {\em Numer. Math.}, 84(3):375--393, 2000.

\bibitem{CZ06}
H.-D. Cao and X.-P. Zhu.
\newblock A complete proof of the {P}oincar\'e and geometrization
  conjectures---application of the {H}amilton-{P}erelman theory of the {R}icci
  flow.
\newblock {\em Asian J. Math.}, 10(2):165--492, 2006.

\bibitem{CT17}
L.-J. Cheng and A.~Thalmaier.
\newblock Characterization of pinched ricci curvature by functional
  inequalities.
\newblock {\em The Journal of Geometric Analysis}, Aug 2017.

\bibitem{CL03}
B.~Chow and F.~Luo.
\newblock Combinatorial {R}icci flows on surfaces.
\newblock {\em J. Differential Geom.}, 63(1):97--129, 2003.

\bibitem{CEMS01}
D.~Cordero-Erausquin, R.~McCann, and M.~Schmuckenschl\"ager.
\newblock A {R}iemannian interpolation inequality \`a la {B}orell, {B}rascamp
  and {L}ieb.
\newblock {\em Invent. Math.}, 146(2):219--257, 2001.

\bibitem{DKZ}
D.~Dier, M.~Kassmann, and R.~Zacher.
\newblock Discrete versions of the {Li}--{Y}au gradient estimate.
\newblock {\em arXiv:1701.04807}, 2017.

\bibitem{EM12}
M.~Erbar and J.~Maas.
\newblock Ricci curvature of finite {M}arkov chains via convexity of the
  entropy.
\newblock {\em Arch. Ration. Mech. Anal.}, 206(3):997--1038, 2012.

\bibitem{EMW}
M.~Erbar, J.~Maas, and M.~Wirth.
\newblock On the geometry of geodesics in discrete optimal transport.
\newblock {\em arXiv:1805.06040}, 2018.

\bibitem{For03}
R.~Forman.
\newblock Bochner's method for cell complexes and combinatorial {R}icci
  curvature.
\newblock {\em Discrete {\&} Computational Geometry}, 29(3):323--374, Feb 2003.

\bibitem{GLM17}
W.~{Gangbo}, W.~{Li}, and C.~{Mou}.
\newblock {Geodesic of minimal length in the set of probability measures on
  graphs}.
\newblock {\em ArXiv e-prints}, December 2017.

\bibitem{GM12}
N.~Gigli and J.~Maas.
\newblock Gromov-{H}ausdorff convergence of discrete transportation metrics.
\newblock {\em SIAM J. Math. Anal.}, 45(2):879--899, 2013.

\bibitem{GKM}
P.~Gladbach, E.~Kopfer, and J.~Maas.
\newblock Scaling limits of discrete optimal transport.
\newblock {\em preprint}, 2018.

\bibitem{Glick05}
D.~Glickenstein.
\newblock A combinatorial {Y}amabe flow in three dimensions.
\newblock {\em Topology}, 44(4):791 -- 808, 2005.

\bibitem{H82}
R.~Hamilton.
\newblock Three-manifolds with positive {R}icci curvature.
\newblock {\em J. Differential Geom.}, 17(2):255--306, 1982.

\bibitem{H95}
R.~Hamilton.
\newblock The formation of singularities in the {R}icci flow.
\newblock In {\em Surveys in differential geometry, {V}ol.\ {II} ({C}ambridge,
  {MA}, 1993)}, pages 7--136. Int. Press, Cambridge, MA, 1995.

\bibitem{KL08}
B.~Kleiner and J.~Lott.
\newblock Notes on {P}erelman's papers.
\newblock {\em Geom. Topol.}, 12(5):2587--2855, 2008.

\bibitem{KL17}
B.~Kleiner and J.~Lott.
\newblock Singular {R}icci flows {I}.
\newblock {\em Acta Math.}, 219(1):65--134, 2017.

\bibitem{sturm2016}
E.~Kopfer and K.-T. Sturm.
\newblock {Heat flows on Time-dependent Metric Measure Spaces and Super-Ricci
  Flows}.
\newblock {\em arXiv:1611.02570}, 2017.

\bibitem{KS18+}
E.~Kopfer and K.-Th. Sturm.
\newblock Super {R}icci flows and functional inequalities.
\newblock {\em in progress}, 2018.

\bibitem{LLY11}
Y.~Lin, L.~Lu, and S.-T. Yau.
\newblock Ricci curvature of graphs.
\newblock {\em Tohoku Math. J. (2)}, 63(4):605--627, 2011.

\bibitem{LV09}
J.~Lott and C.~Villani.
\newblock Ricci curvature for metric-measure spaces via optimal transport.
\newblock {\em Ann. Math. (2)}, 169(3):903--991, 2009.

\bibitem{Ma11}
J.~Maas.
\newblock Gradient flows of the entropy for finite {M}arkov chains.
\newblock {\em J. Funct. Anal.}, 261(8):2250--2292, 2011.

\bibitem{MT10}
R.~McCann and P.~Topping.
\newblock Ricci flow, entropy and optimal transportation.
\newblock {\em Amer. J. Math.}, 132(3):711--730, 2010.

\bibitem{Mie11a}
A.~Mielke.
\newblock A gradient structure for reaction-diffusion systems and for
  energy-drift-diffusion systems.
\newblock {\em Nonlinearity}, 24(4):1329--1346, 2011.

\bibitem{Mie11b}
A.~Mielke.
\newblock Geodesic convexity of the relative entropy in reversible {M}arkov
  chains.
\newblock {\em Calc. Var. Partial Differential Equations, Online first}, 2012.

\bibitem{YauAL14a}
W.~A. Miller, J.~R. McDonald, P.~M. Alsing, D.~X. Gu, and S.-T. Yau.
\newblock Simplicial {R}icci flow.
\newblock {\em Comm. Math. Phys.}, 329(2):579--608, 2014.

\bibitem{MT07}
J.~Morgan and G.~Tian.
\newblock {\em Ricci flow and the {P}oincar\'e conjecture}, volume~3 of {\em
  Clay Mathematics Monographs}.
\newblock American Mathematical Society, Providence, RI; Clay Mathematics
  Institute, Cambridge, MA, 2007.

\bibitem{Oll09}
Y.~Ollivier.
\newblock Ricci curvature of {M}arkov chains on metric spaces.
\newblock {\em J. Funct. Anal.}, 256(3):810--864, 2009.

\bibitem{P02}
G.~Perelman.
\newblock The entropy formula for the {R}icci flow and its geometric
  applications.
\newblock {\em arXiv:0211.159}, 2002.

\bibitem{P03}
G.~Perelman.
\newblock Finite extinction time for the solutions to the {R}icci flow on
  certain three-manifolds.
\newblock {\em arXiv:0307.245}, 2003.

\bibitem{P03b}
G.~Perelman.
\newblock Ricci flow with surgery on three-manifolds.
\newblock {\em arXiv:0303.109}, 2003.

\bibitem{HN15}
Haslhofer R. and A.~Naber.
\newblock Weak solutions for the {R}icci flow i.
\newblock {\em arXiv:1504.00911}, 2015.

\bibitem{vRS05}
M.-K.~von Renesse and K.-Th. Sturm.
\newblock Transport inequalities, gradient estimates, entropy, and {R}icci
  curvature.
\newblock {\em Comm. Pure Appl. Math.}, 58(7):923--940, 2005.

\bibitem{Tan15}
R.~Sandhu, T.~Georgiou, E.~Reznik, L.~Zhu, I.~Kolesov, Y.~Senbabaoglu, and
  A.~Tannenbaum.
\newblock Graph curvature for differentiating cancer networks.
\newblock {\em Nature Scientific Reports}, 5, 2015.

\bibitem{sturm2015}
K.-T. Sturm.
\newblock {Super Ricci flows for metric measure spaces. I}.
\newblock {\em arXiv:1603.02193}, 2016.

\bibitem{S06}
K.-Th. Sturm.
\newblock On the geometry of metric measure spaces. {I} and {II}.
\newblock {\em Acta Math.}, 196(1):65--177, 2006.

\bibitem{Tri}
N.~G. Trillos.
\newblock Gromov-{H}ausdorff limit of {W}asserstein spaces on point clouds.
\newblock {\em arXiv:1702.03464}, 2017.

\bibitem{Vil09}
C.~Villani.
\newblock {\em Optimal transport, Old and new}, volume 338 of {\em Grundlehren
  der Mathematischen Wissenschaften}.
\newblock Springer-Verlag, Berlin, 2009.

\bibitem{WSJ17}
M.~Weber, E.~Saucan, and J.~Jost.
\newblock Characterizing complex networks with forman-ricci curvature and
  associated geometric flows.
\newblock {\em Journal of Complex Networks}, 5(4):527--550, 2017.

\bibitem{Zen13}
W.~Zeng and D.~X. Gu.
\newblock {\em Ricci flow for shape analysis and surface registration}.
\newblock SpringerBriefs in Mathematics. Springer, New York, 2013.
\newblock Theories, algorithms and applications.

\bibitem{Zha15}
M.~Zhang, W.~Zeng, R.~Guo, F.~Luo, and D.~X. Gu.
\newblock Survey on discrete surface {R}icci flow.
\newblock {\em J. Comput. Sci. Tech.}, 30(3):598--613, 2015.

\end{thebibliography}

 \end{document}